\documentclass[a4paper,11pt]{amsart}
\usepackage{amsmath,amssymb,amsthm}
\usepackage[all]{xy}
\usepackage{bm}
\usepackage{mathrsfs}
\usepackage{enumerate}
\usepackage{subfiles}
\usepackage{amsrefs}
\usepackage{lineno}
\newcommand*\patchAmsMathEnvironmentForLineno[1]{
  \expandafter\let\csname old#1\expandafter\endcsname\csname #1\endcsname
  \expandafter\let\csname oldend#1\expandafter\endcsname\csname end#1\endcsname
  \renewenvironment{#1}
     {\linenomath\csname old#1\endcsname}
     {\csname oldend#1\endcsname\endlinenomath}}
\newcommand*\patchBothAmsMathEnvironmentsForLineno[1]{
  \patchAmsMathEnvironmentForLineno{#1}
  \patchAmsMathEnvironmentForLineno{#1*}}
\AtBeginDocument{
\patchBothAmsMathEnvironmentsForLineno{equation}
\patchBothAmsMathEnvironmentsForLineno{align}
\patchBothAmsMathEnvironmentsForLineno{flalign}
\patchBothAmsMathEnvironmentsForLineno{alignat}
\patchBothAmsMathEnvironmentsForLineno{gather}
\patchBothAmsMathEnvironmentsForLineno{multline}
}
\newcommand{\Hom}{\operatorname{Hom}}

\theoremstyle{definition}
\newtheorem{dfn}{Definition}[section]
\theoremstyle{plain}
\newtheorem{thm}[dfn]{Theorem}
\newtheorem{prop}[dfn]{Proposition}
\newtheorem{lem}[dfn]{Lemma}
\newtheorem{cor}[dfn]{Corollary}
\newtheorem{conj}[dfn]{Conjecture}
\theoremstyle{remark}
\newtheorem{rem}[dfn]{Remark}
\newtheorem{ex}[dfn]{Example}

\newtheorem{constr}[dfn]{Construction}

\begin{document}
\title[Gorenstein terminal Fano threefolds]{Rational curves on Fano threefolds with Gorenstein terminal singularities}
\author{Fumiya Okamura}
\address{The Institute of Science and Engineering, Chuo University, 
1-13-27 Kasuga, Bunkyo-ku, Tokyo 112-8551, Japan}
\email{fokamura941@g.chuo-u.ac.jp}
\begin{abstract}
    We study the spaces of rational curves on Fano threefolds with Gorenstein terminal singularities.
    We generalize the results regarding Geometric Manin's Conjecture for smooth Fano threefolds, including the classification of subvarieties with higher $a$-invariants and Movable Bend-and-Break lemma.
    We also show Geometric Manin's Conjecture for some singular del Pezzo threefolds.
\end{abstract}


\maketitle
\section{Introduction}
Throughout the paper, we work over an algebraically closed field $k$ of characteristic $0$.
We say that a normal projective variety $X$ is Fano if there exists a positive integer $r \in \mathbb{Z}$ such that $-rK_{X}$ is an ample Cartier divisor.
We are interested in rational curves on a Fano variety $X$.
By Mori's Bend-and-Break technique, developed in successive papers \cite{Mori1979}, \cite{Mori1982}, \cite{Miyaoka1986}, \cite{KoMiMo1992c}, \cite{Hacon2007}, it is shown that Fano varieties are uniruled, and more generally if $X$ has only mild singularities such as klt singularities, then $X$ is rationally connected.
These results tell us the importance of studying the spaces of rational curves on $X$.
More specifically, we will study the Hom space $\mathrm{Hom}(\mathbb{P}^{1}, X)$ parametrizing morphisms $f\colon \mathbb{P}^{1}\rightarrow X$, or Kontsevich spaces $\overline{M}_{0,r}(X)$ parametrizing $r$-pointed stable maps on $X$ of genus $0$.
For the notion of Kontsevich spaces, see for example \cite{Fulton1997}.
Our goal is to count the number of irreducible components of the Hom space $\mathrm{Hom}(\mathbb{P}^{1}, X, \alpha)$ parametrizing morphisms $f\colon \mathbb{P}^{1} \rightarrow X$ such that $f_{*}\mathbb{P}^{1} = \alpha$, for each curve class $\alpha \in \overline{\mathrm{Eff}}_{1}(X)_{\mathbb{Z}}$.
This is motivated by Geometric Manin's Conjecture, introduced in \cite{Lehmann2019}.
Geometric Manin's Conjecture predicts the asymptotic formula for the number of irreducible components of $\mathrm{Hom}(\mathbb{P}^{1}, X, \alpha)$ after removing pathological components.
We expect that pathological components can be detected by two geometric invariants $a$ and $b$ defined as follows:
for simplicity, we let $X$ be a smooth projective variety and $L$
be a nef and big divisor on $X$.
Then the $a$-invariant is defined as
\[a(X,L) = \inf \{t\in \mathbb{R} \mid K_{X} + tL\mbox{ is pseudo-effective}\}.\]
When $a(X,L)>0$ (or equivalently, $X$ is uniruled, \cite{Boucksom2013}), then the $b$-invariant is defined as 
\[b(X,L) = \dim \langle F(X,L)\rangle,\]
where $\langle - \rangle$ denotes the linear span and 
\[F(X,L) = \{\alpha \in \mathrm{Nef}_{1}(X) \mid (K_{X} + a(X,L)L)\cdot \alpha = 0\}.\]
These invariants are birational invariants by \cite{Hassett2015}.
Hence we define $a$ and $b$-invariants when $X$ is singular: 
\[a(X, L) := a(\tilde{X}, \phi^{*}L),\quad  b(X, L) := b(\tilde{X}, \phi^{*}L), \] 
where $\phi\colon \tilde{X}\rightarrow X$ is a smooth resolution.
\par
Conceptually, pathological components should come from the contribution of generically finite morphisms $f\colon Y\rightarrow X$ onto the images such that 
\[(a(Y, -f^{*}K_{X}), b(Y, -f^{*}K_{X})) \ge (a(X, -K_{X}), b(X, -K_{X}))\]
in the lexicographic order.
The authors of \cite{Lehmann2019} defined accumulating components as the candidates of pathological components.
For the definition of accumulating components, see \cite{Tanimoto2021}, or Definition \ref{breaking} of this paper (the definition has been slightly modified from \cite{Lehmann2019}).
We also give a precise statement of Geometric Manin's Conjecture in Section 6.1.
The paper \cite{LRT2024nonfree_section} formulate Geometric Manin's Conjecture in most general setting. However, we do not consider this version of the conjecture. 

\subsection*{Main results}
In this paper, we consider ($\mathbb{Q}$-factorial) Gorenstein terminal Fano threefolds $X$ and generalize the results of \cite{Beheshti2022} on smooth Fano threefolds.
One of the main results is the classification of subvarieties $Y$ with higher $a$-invariants: 
\begin{thm}[\S.3]\label{highA}
Let $X$ be a Gorenstein terminal Fano threefold.
Let $Y \subset X$ be a $2$-dimensional subvariety of $X$  satisfying $a(Y, -K_{X}) > 1$ and let $\phi \colon \tilde{Y} \rightarrow Y$ be a smooth resolution.
\begin{enumerate}[(1)]
\item 
If $\kappa (\tilde{Y}, K_{\tilde{Y}} - a(Y, -K_{X})\phi^{*}K_{X}) = 1$, then $Y$ is covered by a family of $-K_{X}$-lines.
\item 
If $\kappa (\tilde{Y}, K_{\tilde{Y}} - a(Y, -K_{X})\phi^{*}K_{X}) = 0$ and $Y$ is normal, then $(Y, -K_{X}|_{Y})$ is isomorphic to either $(\mathbb{P}^{2}, \mathcal{O}(1))$, $(Q, \mathcal{O}(1))$, or $(\mathbb{P}^{2}, \mathcal{O}(2))$, where $Q \subset \mathbb{P}^{3}$ is a smooth quadric or a quadric cone.
Moreover, if $Y$ is a Cartier divisor, then $Y$ is an E2, E3, E4, E5 divisor, or E1 divisor of the form $(\mathbb{P}^{1} \times \mathbb{P}^{1}, \mathcal{O}(1,1))$.
\item 
If $\kappa (\tilde{Y}, K_{\tilde{Y}} - a(Y, -K_{X})\phi^{*}K_{X}) = 0$ and $Y$ is non-normal, then the normalization of $Y$ is isomorphic to either $(\mathbb{P}^{2}, \mathcal{O}(1))$, $(Q, \mathcal{O}(1))$, or $(\mathbb{P}^{2}, \mathcal{O}(2))$, where $Q \subset \mathbb{P}^{3}$ is a smooth quadric or a quadric cone.
If furthermore $-K_{X}$ is very ample and $Y$ is a Cartier divisor, then one of the following holds:
\begin{enumerate}[(a)]
\item 
There is a morphism $\pi \colon X \rightarrow \mathbb{P}^{1}$ whose general fiber is a del Pezzo surface of degree $4$ such that $Y = \pi^{-1}(q)$ for some point $q \in \mathbb{P}^{1}$, 
\item 
There is a birational map $\rho \colon X \dashrightarrow X'$ to a quartic hypersurface in $\mathbb{P}^{4}$ such that $\rho|_{Y}$ is an isomorphism and $\rho(Y)$ is a hyperplane section of $X'$. 
\end{enumerate}
\end{enumerate}
\end{thm}
Theorem \ref{highA} helps us to find components of $\mathrm{Hom}(\mathbb{P}^{1},X)$ whose dimension are higher than the expected dimension.
We also give a partial results on the classification of $a$-covers, i.e., dominant generically finite morphisms $f\colon Y\rightarrow X$ of degree $\ge 2$ such that $a(Y, -f^{*}K_{X}) = 1$: 
\begin{thm}[\S.4]\label{a-covers}
    Let $X$ be a Gorenstein terminal Fano threefold.
    Let $f\colon Y\rightarrow X$ be an $a$-cover with $Y$ smooth. 
    Set $\kappa := \kappa(Y, K_{Y}-f^{*}K_{X}) \in \{0,1,2\}$. 
    \begin{itemize}
        \item When $\kappa = 2$, $f$ is birationally a base change of a family of $-K_{X}$-conics, 
        \item Suppose that $X$ is a factorial del Pezzo threefold of degree at least $2$ (see Definition \ref{Def_dP} for del Pezzo varieties).
        Then $X$ has no $a$-covers with $\kappa = 1$, 
        \item Suppose that $X$ is factorial and admits no E5 contractions.
        Then $X$ has no $a$-covers with $\kappa = 0$.
    \end{itemize}
\end{thm}
By these classification results, we can narrow down the candidates for accumulating components to a certain extent.
Once one can remove accumulating components, Movable Bend-and-Break plays a very efficient role in the induction step of Geometric Manin's Conjecture.
The third main theorem is as follows:
\begin{thm}[Movable Bend-and-Break, \S.5]\label{MBB_noE5}
    Let $X$ be a terminal factorial Fano threefold without E5 contractions.
    Then any component of $\overline{M}_{0,0}(X)$ generically parametrizing free stable maps of anticanonical degree at least $5$ contains a stable map $f\colon C\rightarrow X$ such that 
    \begin{itemize}
        \item the domain $C$ consists of two irreducible components $C_{1}, C_{2}$, and 
        \item each restriction $f|_{C_{i}}\colon C_{i}\rightarrow X$ is a free rational curve.
    \end{itemize}
\end{thm}
In the last section, we consider cubic hypersurfaces in $\mathbb{P}^{4}$ and complete intersections of two quadrics in $\mathbb{P}^{5}$ with terminal factorial singularities.
In these cases, we study the irreducibility of the spaces of lines and conics.
As a result, we conclude Geometric Manin's Conjecture: 
\begin{thm}[\S.6]\label{GMC_dP}
    Let $X$ be a factorial del Pezzo threefold with an ample generator $H$.
    Suppose that $H^{3} \ge 3$.
    Then Geometric Manin's Conjecture holds for $X$ with $c = 1$ and $\alpha$ is the class of $H$-conics, that is, after removing accumulating components, for each $d\ge 2$, the space $\mathrm{Hom}(\mathbb{P}^{1}, X,d)$ parametrizing morphisms $f\colon \mathbb{P}^{1}\rightarrow X$ such that $H\cdot f_{*}\mathbb{P}^{1} = d$ contains exactly one irreducible component.
\end{thm}

Geometric Manin's Conjecture is confirmed for smooth del Pezzo threefolds of $H^{3} \ge 3$ in \cite{Lehmann2019}.
Moreover, any factorial del Pezzo threefold with $H^{3} \ge 5$ is smooth by \cite{Prokhorov2013a}*{Lemma 3.3} and \cite{Prokhorov2017}. 
However, there are non-smooth factorial del Pezzo threefolds of $H^{3} = 3, 4$: any cubic threefold with at most $3$ nodes is factorial by \cite{PolizziRapagnettaSabatino2014factorial}*{Theorem 4.4}, and any complete intersection of two quadrics in $\mathbb{P}^{5}$ with a single node is factorial by \cite{Kosta2009factorialCI}*{Theorem 1.1}.
Thus we focus on cubic threefolds and complete intersections of two quadrics in $\mathbb{P}^{5}$ with terminal factorial singularities. 

\subsection*{Outline}
The paper is organized as follows.
In Section 2, we collect preliminary results on free rational curves, Gorenstein terminal threefolds, and $a$, $b$-invariants.
In Section 3, we give a classification of subvarieties $Y$ of a Gorenstein terminal Fano threefold $X$ such that $a(Y, -K_{X}|_{Y}) > 1$.
In Section 4, we give a partial result on the classification of $a$-covers $f\colon Y\rightarrow X$ with respect to the Iitaka dimension $\kappa(Y, K_{Y}-f^{*}K_{X}) \in \{0,1,2\}$.
In Section 5, we prove Movable Bend-and-Break for terminal factorial Fano threefolds without E5 contractions.
In the last section, we state Geometric Manin's Conjecture and give a list of known results on Geometric Manin's Conjecture. 
Then we prove Geometric Manin's Conjecture for singular cubic and quartic del Pezzo threefolds.

\section*{Acknowledgments}
The author would like to thank Sho Tanimoto for suggesting the topic treated in the paper and for his continuous support.
The author would like to thank the anonymous referees for the invaluable comments and suggestions.
\par
The author was partially supported by JST FOREST program Grant number JPMJFR212Z and JSPS Bilateral Joint Research Projects Grant number JPJSBP120219935.

\section{Preliminaries}
\subsection{Free rational curves}
For a projective variety $X$ over $k$, we write the set of numerical classes of $\mathbb{R}$-Cartier divisors (resp. $\mathbb{R}$-$1$-cycles) as $N^{1}(X)$ (resp. $N_{1}(X)$).
These are $\mathbb{R}$-vector spaces of dimension $\rho(X) < \infty$, the Picard number of $X$, and are considered as Euclidean spaces.
Let $\overline{\mathrm{Eff}}^{1}(X)$ be the closed convex cone consisting of the classes of pseudo-effective $\mathbb{R}$-Cartier divisors on $X$, and $\mathrm{Nef}^{1}(X)$ the closed convex cone consisting of the classes of nef $\mathbb{R}$-Cartier divisors on $X$.
Similarly, let $\overline{\mathrm{Eff}}_{1}(X)$ be the closed convex cone consisting of the classes of pseudo-effective $\mathbb{R}$-$1$-cycles on $X$, and $\mathrm{Nef}_{1}(X)$ the closed convex cone consisting of the classes of nef $\mathbb{R}$-$1$-cycles on $X$.
The intersection pairing $N^{1}(X) \times N_{1}(X) \rightarrow \mathbb{R}$ induces the duality of the cones $\overline{\mathrm{Eff}}^{1}(X) = \mathrm{Nef}_{1}(X)^{\vee}$ and $\mathrm{Nef}^{1}(X) = \overline{\mathrm{Eff}}_{1}(X)^{\vee}$.
\par
For an integral curve class $\alpha \in \overline{\mathrm{Eff}}_{1}(X)$, let $\Hom (\mathbb{P}^{1}, X, \alpha)$ be the quasi-projective scheme parametrizing morphisms $f\colon \mathbb{P}^{1}\rightarrow X$ such that $f_{*}\mathbb{P}^{1} = \alpha$.
We often call a member of $\Hom (\mathbb{P}^{1}, X, \alpha)$ a rational curve on $X$ representing $\alpha$.
Set $\mathrm{Hom}(\mathbb{P}^{1}, X) := \bigsqcup_{\alpha \in \overline{\mathrm{Eff}}_{1}(X)} \Hom (\mathbb{P}^{1}, X, \alpha)$.
We define free and very free rational curves on a (not necessarily smooth) projective variety.
\begin{dfn}
Let $X$ be a projective variety.
We say that a rational curve $f\in \mathrm{Hom}(\mathbb{P}^{1}, X)$ is free (resp. very free) if 
\begin{enumerate}[(i)]
\item
the image $f(\mathbb{P}^{1})$ does not meet with the singular locus of $X$, and 
\item 
$H^{1}(\mathbb{P}^{1}, f^{*}T_{X}(-1)) = 0$ (resp. $H^{1}(\mathbb{P}^{1}, f^{*}T_{X}(-2)) = 0$), or equivalently, $f^{*}T_{X}$ (resp. $f^{*}T_{X}(-1))$ is generated by global sections.
\end{enumerate}
\end{dfn}
By \cite{Hartshorne1977}*{V. Exercise 2.6}, any locally free sheaf on $\mathbb{P}^{1}$ can be written as a direct sum of invertible sheaves.
Hence the freeness (resp. very freeness) is equivalent to the fact that $f^{*}T_{X}$ has the form 
\[f^{*}T_{X} = \mathcal{O}(a_{1}) \oplus \dots \oplus \mathcal{O}(a_{\dim(X)}),\]
where $a_{1}, \dots, a_{\dim(X)} \in \mathbb{Z}$ are all non-negative (resp. positive).
\begin{thm}[\cite{Kollar1996}*{II.3.11}]\label{Kollar dom iff free}
Let $X$ be a smooth projective variety and $\alpha \in \overline{\mathrm{Eff}}_{1}(X)_{\mathbb{Z}}$.
Then there is a proper closed subset $V \subset X$ which contains all non-free rational curves representing $\alpha$.
\end{thm}
\par
For a projective variety and an integral curve class $\alpha \in \overline{\mathrm{Eff}}^{1}(X)$, let $\overline{M}_{0, r}(X, \alpha)$ be the Kontsevich space parametrizing $r$-pointed stable maps $f\colon C\rightarrow X$ of genus $0$ such that $f_{*}C = \alpha$ (For the definition and construction, see for example \cite{Fulton1997}).
Set $\overline{M}_{0, r}(X) = \bigsqcup_{\alpha \in \overline{\mathrm{Eff}}^{1}(X)} \overline{M}_{0, r}(X, \alpha)$.
Let $\overline{\mathrm{Rat}}_{r}(X, \alpha)$ (resp. $\overline{\mathrm{Free}}_{r}(X, \alpha)$) be the union of components of $\overline{M}_{0, r}(X, \alpha)$ generically parametrizing stable maps with irreducible domains (resp. free stable maps).
When $r = 0$, we simply write them as $\overline{\mathrm{Rat}}(X, \alpha)$ and $\overline{\mathrm{Free}}(X, \alpha)$.

\subsection{Gorenstein terminal threefolds}
Let $X$ be a Gorenstein terminal threefold.
By \cite{KollarMori1998}*{Corollary 5.38}, $X$ has only isolated cDV singularities. 
Moreover, it is known that $X$ is local complete intersection (e.g., in the proof of \cite{Lehmann2023}*{Lemma 2.2}).
Hence by deformation theory, we obtain the lower bound for the dimension of $\mathrm{Hom}(\mathbb{P}^{1}, X)$. 
\begin{prop}[\cite{Kollar1996}*{II.1.3}, \cite{Lehmann2023}*{Lemma 2.2}]\label{Homdim_Gor}
Let $X$ be a Gorenstein terminal threefold and let $\alpha \in \overline{\mathrm{Eff}}_{1}(X)$. 
For $[f]\in \Hom (\mathbb{P}^{1}, X, \alpha)$, we have 
\[\dim_{[f]} \Hom (\mathbb{P}^{1}, X, \alpha) \ge -K_{X}\cdot \alpha + \dim X.\]
\end{prop}
\begin{lem}[\cite{Lehmann2023}*{Lemma 2.3}]\label{free_iff_dominant}
Let $X$ be a Gorenstein terminal threefold and $M \subset \overline{\mathrm{Rat}}(X)$ be a component.
Then $M \subset \overline{\mathrm{Free}}(X)$ if and only if $M$ is a dominant component, that is, the associated evaluation morphism is dominant.
\end{lem}
For Gorenstein terminal threefolds, $\mathbb{Q}$-factoriality and factoriality are equivalent: 
\begin{lem}[\cite{Kawamata1988}*{Lemma 5.1}]
Let $X$ be a Gorenstein terminal threefold and let $D$ be a Weil divisor.
Then $D$ is $\mathbb{Q}$-Cartier if and only if it is Cartier.
In particular, $\mathbb{Q}$-factorial Gorenstein terminal threefolds are factorial.
\end{lem}

\begin{thm}[\cite{Cutkosky1988}]\label{contractible}
Let $X$ be a terminal factorial threefold and $\phi\colon X\rightarrow X'$ be an elementary divisorial contraction with the exceptional divisor $E$.
Then one of the following holds:
\begin{enumerate}[(E1)]
\item $\phi$ is the blow up along a local complete intersection curve $C$, and $\phi|_{E} \colon E \rightarrow C$ is a ruled structure whose general fiber is a $-K_{X}$-line.
\item $\phi$ is the blow up at a smooth point of $X'$, and $(E, -K_{X}|_{E}) \cong (\mathbb{P}^{2}, \mathcal{O}(2))$.
\item $\phi$ is the blow up at a point of $X'$ which is analytically isomorphic to $0 \in (x^{2} + y^{2} + z^{2} + w^{2} = 0) \subset \mathbb{A}^{4}$, and $(E, -K_{X}|_{E}) \cong (Q, \mathcal{O}(1))$, where $Q$ is a smooth quadric in $\mathbb{P}^{3}$.
\item $\phi$ is the blow up at a point of $X'$ which is analytically isomorphic to $0 \in (x^{2} + y^{2} + z^{2} + w^{n} = 0) \subset \mathbb{A}^{4}$ for some $n \ge 3$, and $(E, -K_{X}|_{E}) \cong (Q, \mathcal{O}(1))$, where $Q$ is a quadric cone in $\mathbb{P}^{3}$.
\item $\phi$ is the blow up at a point of $X'$ which is analytically isomorphic to $0 \in \mathbb{A}^{3}/\iota$, where $\iota$ is the involution given by $(x, y, z) \mapsto (-x, -y, -z)$, and $(E, -K_{X}|_{E}) \cong (\mathbb{P}^{2}, \mathcal{O}(1))$.
\end{enumerate}
Note that Gorenstein terminal threefolds do not have small contractions.
\end{thm}

\subsection{$a$, $b$-invariants}
In this subsection, we define two invariants which appear in Geometric Manin's Conjecture.
\begin{dfn}[\cite{Hassett2015}*{Definition 2.2}]
Let $X$ be a $\mathbb{Q}$-factorial canonical projective variety with a nef and big $\mathbb{Q}$-divisor $L$.
We define the $a$-invariant (or the Fujita invariant) by
\[a(X, L) = \inf\{t \in \mathbb{R} \mid K_{X} + tL \in \overline{\mathrm{Eff}}^{1}(X)\}.\]
This is a birational invariant in this category by \cite{Hassett2015}*{Proposition 2.7}.
Hence we define the $a$-invariant for other projective varieties by 
\[a(X, L) := a(\bar{X}, \phi^{*}L), \]
where $\phi\colon \bar{X}\rightarrow X$ is a birational morphism from a $\mathbb{Q}$-factorial canonical projective variety $\bar{X}$ (e.g., one can take a smooth resolution).
\end{dfn}
By \cite{Boucksom2013}, $a(X, L) > 0$ if and only if $X$ is uniruled.
In particular,  for any Fano variety, one sees that $a(X, L)$ is positive.

It is important to compare the a-invariants of subvarieties.
In particular, there is little amount of subvarieties with higher $a$-invariants:
\begin{thm}[cf. \cite{Lehmann2019}*{Theorem 3.3}]\label{clsd}
Let $X$ be a uniruled projective variety with a nef and big $\mathbb{Q}$-Cartier divisor $L$.
Then the closure $V$ of the union of all subvarieties $Y$ satisfying $a(Y, L) > a(X, L)$ is a proper closed subset of $X$.
Moreover, any irreducible component of $V$ not contained in the singular locus also satisfies the same inequality.
\end{thm}
\begin{proof}
    Take a smooth resolution $\phi\colon \tilde{X}\rightarrow X$. 
    Let $Y\subset X$ be a subvariety not contained in the singular locus $\mathrm{Sing}(X)$ such that $a(Y, L)>a(X,L)$. 
    Then the strict transform $\tilde{Y}$ satisfies $a(\tilde{Y}, \phi^{*}L) = a(Y, L) > a(X, L) = a(\tilde{X}, \phi^{*}L)$. 
    Conversely, if $\tilde{Y} \subset \tilde{X}$ is a subvariety not contained in the exceptional locus $\mathrm{Exc}(\phi)$ such that $a(\tilde{Y}, \phi^{*}L) > a(\tilde{X}, \phi^{*}L)$, then the image $Y$ also satisfies $a(Y, L) > a(X, L)$. 
    By \cite{Lehmann2019}*{Theorem 3.3}, the union $\tilde{V}$ of all subvarieties with $a(\tilde{Y}, \phi^{*}L) > a(\tilde{X}, \phi^{*}L)$ is a proper closed subset of $\tilde{X}$, and each irreducible component $\tilde{V}_{i}$ of $\tilde{V}$ also satisfies $a(\tilde{V}_{i}, \phi^{*}L) > a(\tilde{X}, \phi^{*}L)$. 
    Thus we see that $V\subset \phi(\tilde{V}) \cup \mathrm{Sing}(X) \subsetneq X$ and that each component $V_{i}$ of $V$ not contained in $\mathrm{Sing}(X)$ is the birational image of some component $\tilde{V}_{i}$ of $\tilde{V}$, hence satisfies $a(V_{i}, L) > a(X, L)$. 
\end{proof}

\begin{dfn}[\cite{Hassett2015}*{Definition 2.8}]
Let $X$ be a $\mathbb{Q}$-factorial terminal uniruled projective variety with a nef and big $\mathbb{Q}$-divisor $L$.
Set 
\[F(X, L) = \{\alpha \in \mathrm{Nef}_{1}(X) \mid (K_{X} + a(X, L)L)\cdot \alpha = 0\}.\]
Then we define the $b$-invariant by 
\[b(X, L) = \dim \langle F(X, L)\rangle,\]
where $\langle - \rangle$ denotes the linear span.
This is a birational invariant in this category by \cite{Hassett2015}*{Proposition 2.10}.
Hence we define the $b$-invariant for other uniruled projective varieties by 
\[b(X, L) := b(\bar{X}, \phi^{*}L), \]
where $\phi \colon \bar{X}\rightarrow X$ is a birational morphism from a $\mathbb{Q}$-factorial terminal uniruled projective variety $\bar{X}$ (e.g., one can take a smooth resolution).
\end{dfn}

\begin{rem}
    For the special case $L = -K_{X}$, we have $a(X, -K_{X}) = 1$, $F(X, -K_{X}) = \mathrm{Nef}_{1}(X)$, and $b(X, -K_{X}) = \rho(X)$.
\end{rem}

\begin{lem}[\cite{Sengupta2021}, Lemma 2.4]\label{a-inv_dominant_finite}
    Let $X$ be a projective variety with a nef and big $\mathbb{Q}$-Cartier divisor $L$.
    Let $f\colon Y\rightarrow X$ be a dominant generically finite morphism from a projective variety $Y$.
    Then we have $a(Y, f^{*}L) \le a(X, L)$.
\end{lem}

\begin{dfn}\label{Def_covers}
    Let $X$ be a projective variety with a nef and big $\mathbb{Q}$-Cartier divisor $L$.
    Let $f\colon Y\rightarrow X$ be a morphism from a projective variety $Y$.
    \begin{enumerate}[(1)]
        \item We say that $f$ is a thin morphism if $f$ is a generically finite morphism onto its image but not birational to $X$.
        \item We say that $f$ is an $a$-cover if $f$ is a dominant thin morphism such that $a(Y, f^{*}L) = a(X, L)$.
        \item Suppose furthermore that $X$ and $Y$ are $\mathbb{Q}$-factorial terminal uniruled projective varieties.
        Then we say that $f$ is a face contracting morphism if $f$ is an $a$-cover and the induced map $f_{*}\colon F(Y, f^{*}L)\rightarrow F(X,L)$ is not injective.
    \end{enumerate}
\end{dfn}

\section{Subvarieties with higher $a$-invariants}
\begin{prop}\label{NonDom-highA}
Let $X$ be a Gorenstein terminal Fano threefold. 
Take an irreducible component $M\subset \mathrm{Hom}(\mathbb{P}^{1}, X, \alpha)$, where $\alpha \in \overline{\mathrm{Eff}}_{1}(X)_{\mathbb{Z}}$.
Let $\mathcal{U} \rightarrow M$ be the universal family with the evaluation morphism $s\colon \mathcal{U} \rightarrow X$.
If the closure of the image $Y = \overline{s(\mathcal{U})}$ is a proper closed subvariety of $X$, then we have $a(Y, -K_{X}) > 1$.
\end{prop}
\begin{proof} 
    One can run the same argument as in \cite{Lehmann2019}*{Proposition 4.2}.
    We give a proof for completeness.
    Take a smooth resolution $\phi\colon \tilde{Y}\rightarrow Y$.
    Let $\tilde{M} \subset \mathrm{Hom}(\mathbb{P}^{1}, \tilde{Y}, \beta)$ be an irreducible component containing strict transforms of general members of $M$, where $\beta \in N_{1}(\tilde{Y})$ satisfies $\phi_{*}\beta = \alpha$. 
    By Proposition \ref{Homdim_Gor}, the dimension of $M$ is at least $-K_{X}\cdot \alpha + \dim X$.
    On the other hand, since $\tilde{M}$ generically parametrizes free morphisms to $\tilde{Y}$ by Theorem \ref{Kollar dom iff free}, one has $\dim 
    \tilde{M} = -K_{\tilde{Y}}\cdot \beta + \dim \tilde{Y}$.
    Hence we obtain that $(K_{\tilde{Y}} - \phi^{*}K_{X})\cdot \beta = K_{\tilde{Y}}\cdot \beta -K_{X}\cdot \alpha \le \dim \tilde{Y} - \dim X < 0$.
    This means that $K_{\tilde{Y}} - \phi^{*}K_{X} \notin \overline{\mathrm{Eff}}^{1}(\tilde{Y})$, and $a(Y, -K_{X}) > 1$.
\end{proof}

Let $X$ be a Gorenstein terminal Fano threefold.
We classify subvarieties of $X$ with higher $a$-invariants.
Note that the dimension of each component of $\overline{\mathrm{Rat}}(X)$ is positive, hence the closed subset $V \subset X$ as in Theorem \ref{clsd} is purely $2$-dimensional.
\newtheorem*{MainThm1}{\rm\bf Theorem \ref{highA}}
Let us recall the statement of Theorem \ref{highA}:
\begin{MainThm1}
Let $X$ be a Gorenstein terminal Fano threefold.
Let $Y \subset X$ be a $2$-dimensional subvariety of $X$  satisfying $a(Y, -K_{X}) > 1$ and let $\phi \colon \tilde{Y} \rightarrow Y$ be a smooth resolution.
\begin{enumerate}[(1)]
\item 
If $\kappa (\tilde{Y}, K_{\tilde{Y}} - a(Y, -K_{X})\phi^{*}K_{X}) = 1$, then $Y$ is covered by a family of $-K_{X}$-lines.
\item 
If $\kappa (\tilde{Y}, K_{\tilde{Y}} - a(Y, -K_{X})\phi^{*}K_{X}) = 0$ and $Y$ is normal, then $(Y, -K_{X}|_{Y})$ is isomorphic to either $(\mathbb{P}^{2}, \mathcal{O}(1))$, $(Q, \mathcal{O}(1))$, or $(\mathbb{P}^{2}, \mathcal{O}(2))$, where $Q \subset \mathbb{P}^{3}$ is a smooth quadric or a quadric cone.
Moreover, if $Y$ is a Cartier divisor, then $Y$ is an E2, E3, E4, E5 divisor, or E1 divisor of the form $(\mathbb{P}^{1} \times \mathbb{P}^{1}, \mathcal{O}(1,1))$.
\item 
If $\kappa (\tilde{Y}, K_{\tilde{Y}} - a(Y, -K_{X})\phi^{*}K_{X}) = 0$ and $Y$ is non-normal, then the normalization of $Y$ is isomorphic to either $(\mathbb{P}^{2}, \mathcal{O}(1))$, $(Q, \mathcal{O}(1))$, or $(\mathbb{P}^{2}, \mathcal{O}(2))$, where $Q \subset \mathbb{P}^{3}$ is a smooth quadric or a quadric cone.
If furthermore $-K_{X}$ is very ample and $Y$ is a Cartier divisor, then one of the following holds:
\begin{enumerate}[(a)]
\item 
There is a morphism $\pi \colon X \rightarrow \mathbb{P}^{1}$ whose general fiber is a del Pezzo surface of degree $4$ such that $Y = \pi^{-1}(q)$ for some point $q \in \mathbb{P}^{1}$, 
\item 
There is a birational map $\rho \colon X \dashrightarrow X'$ to a quartic hypersurface in $\mathbb{P}^{4}$ such that $\rho|_{Y}$ is an isomorphism and $\rho(Y)$ is a hyperplane section of $X'$. 
\end{enumerate}
\end{enumerate}
\end{MainThm1}
\begin{rem}
\begin{enumerate}[(i)]
\item 
When surfaces $Y_{1}, Y_{2}$ as in (2) are not Cartier, then $Y_{1}$ and $Y_{2}$ are not necessarily disjoint.
See Example \ref{nonCartier}.
\item
\cite{Beheshti2022}*{Theorem 4.1} shows that any smooth Fano threefold does not contain any non-normal surfaces as in (3a) and (3b).
We give examples of surfaces of each type in Example \ref{typeI} and \ref{typeII} below.
\end{enumerate}
\end{rem}
\begin{proof}[Proof of Theorem \ref{highA}]
We mainly follow the argument as in the proof of \cite{Beheshti2022}*{Theorem 4.1}. 
Suppose that $\kappa (\tilde{Y}, K_{\tilde{Y}} - a(Y, -K_{X})\phi^{*}K_{X}) = 1$ and let $\psi \colon \tilde{Y} \rightarrow B$ be the Iitaka fibration.
Then the general fiber is a curve $C$ with $a(C, -\phi^{*}K_{X}) > a(X, -K_{X}) = 1$.
Hence we have
\[1 < a(C, -\phi^{*}K_{X}) = \frac{\deg{K_{C}}}{\phi^{*}K_{X}\cdot C}, \]
which yields that $C$ is rational and $-K_{X}\cdot \phi_{*}C = 1$, as required.
\par
Suppose that $\kappa (\tilde{Y}, K_{\tilde{Y}} - a(Y, -K_{X})\phi^{*}K_{X}) = 0$.
We first consider the case when $Y$ is normal.
By \cite{Lehmann2021a}*{Lemma 5.3}, $Y$ has at worst canonical singularities and $K_{Y} \sim_{\mathbb{Q}} a(Y, -K_{X})K_{X}|_{Y}$. 
Then by \cite{Hoering2010}*{Proposition 1.3}, $(Y, -K_{X})$ is isomorphic to either $(\mathbb{P}^{2}, \mathcal{O}(1))$ or $(\mathbb{P}^{2}, \mathcal{O}(2))$ or $(Q, \mathcal{O}(1))$, where $Q$ is a smooth quadric or a quadric cone in $\mathbb{P}^{3}$.
We assume furthermore that $Y$ is Cartier. 
Since $(K_{X} + Y)|_{Y} \sim_{\mathbb{Q}} a(Y, -K_{X})K_{X}|_{Y}$, we see that $Y|_{Y} \sim_{\mathbb{Q}} (a(Y, -K_{X}) - 1)K_{X}|_{Y}$ is antiample.
Hence $Y$ is a contractible divisor.
Then by Theorem \ref{contractible}, $Y$ has type E5 if $(Y, -K_{X}) \cong (\mathbb{P}^{2}, \mathcal{O}(1))$, type E2 if $(Y, -K_{X}) \cong (\mathbb{P}^{2}, \mathcal{O}(2))$, and type E1, E3, E4 if $(Y, -K_{X}) \cong (Q, \mathcal{O}(1))$.
\par
Suppose that $Y$ is non-normal.
Let $\nu \colon \bar{Y} \rightarrow Y$ be the normalization.
The same argument as above shows that $(\bar{Y}, -\nu^{*}K_{X})$ is isomorphic to either $(\mathbb{P}^{2}, \mathcal{O}(1))$ or $(\mathbb{P}^{2}, \mathcal{O}(2))$ or $(Q, \mathcal{O}(1))$.
From now on, we assume that $-K_{X}$ is very ample and $Y$ is Cartier.
Since $-K_{X}$ is very ample, $\nu$ must be a morphism defined by a strict sublinear system of $|-\nu^{*}K_{X}|$.
Hence the possibilities are $(\bar{Y}, -\nu^{*}K_{X}) \cong (\mathbb{P}^{2}, \mathcal{O}(2))$ and the sublinear system $V \subset |\mathcal{O}(2)|$ has dimension $3$ or $4$.
\begin{enumerate}[(a)]
\item 
If $\dim V = 4$, then $Y$ is isomorphic to a complete intersection of two quadrics in $\mathbb{P}^{4}$ which is singular along a line.
In particular, $K_{Y}$ is antiample and $K_{Y}^{2} = 4$.
Thus $\mathcal{O}_{\mathbb{P}^{2}}(2) \sim -\nu^{*}K_{Y} \sim -\nu^{*}(K_{X} + Y)|_{Y}$, and hence $Y|_{Y}$ is trivial.
Moreover, there is an exact sequence
\[0 \rightarrow H^{0}(X, \mathcal{O}_{X}) \rightarrow H^{0}(X, \mathcal{O}_{X}(Y)) \rightarrow H^{0}(Y, \mathcal{O}_{Y}) \rightarrow 0\]
of global sections by Kawamata-Viehweg vanishing theorem.
Hence $h^{0}(X, \mathcal{O}_{X}(Y)) = 2$ and $|Y|$ defines a morphism $\pi \colon X \rightarrow \mathbb{P}^{1}$ such that $\pi^{-1}(q) \cong Y$ for some point $q \in \mathbb{P}^{1}$.
\item 
If $\dim V = 3$, then $Y$ is isomorphic to a quartic surface in $\mathbb{P}^{3}$.
By the adjunction formula, $K_{Y}$ is trivial, hence $Y|_{Y} \sim -K_{X}|_{Y}$ is very ample.
Then we have $h^{0}(\mathcal{O}_{Y}(Y)) = 4$ and there is an exact sequence 
\[0 \rightarrow H^{0}(X, \mathcal{O}_{X}) \rightarrow H^{0}(X, \mathcal{O}_{X}(Y)) \rightarrow H^{0}(Y, \mathcal{O}_{Y}(Y)) \rightarrow 0\]
of global sections by Kawamata-Viehweg vanishing theorem.
Hence $h^{0}(X, \mathcal{O}_{X}(Y)) = 5$ and $|Y|$ defines a morphism $\rho \colon X \rightarrow \mathbb{P}^{4}$.
Since $Y^{3} = (Y|_{Y})^{2} = (-K_{X}|_{Y})^{2} =4$, we  see that $\rho$ is a birational map to a quartic threefold.
Since $Y|_{Y}$ is very ample, the restriction $\rho|_{Y}$ is an isomorphism, as required.
\end{enumerate}
\end{proof}

\begin{ex}\label{nonCartier}
Let $f\colon X\rightarrow \mathbb{P}^{3}$ be a double cover ramified along a sextic $B \in |\mathcal{O}_{\mathbb{P}^{3}}(6)|$.
Moreover we take $B$ of the form 
\[B = V(x_{0}F + G^{2}),\]
where $F, G \in k[x_{0}, \dots, x_{3}]$ are homogeneous polynomials of degree $5$ and $3$ respectively.
Then $X$ is defined by $y^{2} = x_{0}F + G^{2}$, where $y$ is a  variable of weight $3$.
When $F$ and $G$ are general, $B$ has only $A_{1}$ singularities, hence $X$ has only $cA_{1}$ singularities.
Now we take a section by a hyperplane $H= V(x_{0})$.
Then we obtain two divisors 
\[D_{1} = V(x_{0}, y-G), \quad D_{2} = V(x_{0}, y+G)\]
which satisfy $(D_{i}, -K_{X}) \cong (\mathbb{P}^{2}, \mathcal{O}(1))$ and $D_{1} \cap D_{2} \neq \emptyset$. 
If these are $\mathbb{Q}$-Cartier, then these are contractible divisors by Theorem \ref{highA}.(2). 
Then $D_{1}$ and $D_{2}$ must be disjoint, a contradiction. 
Thus these are not $\mathbb{Q}$-Cartier.
\end{ex}
\begin{ex}\label{typeI}
Consider the morphism $\phi \colon \mathbb{P}^{2} \rightarrow \mathbb{P}^{4}$ defined by 
\[(x:y:z)\mapsto (z^{2}:-xy:-y^{2}+yz:x^{2}:xy-xz).\]
The image $S$ is a complete intersection of two quadrics
\[S = V(X_{1}X_{4}-X_{2}X_{3},(X_{1}+X_{4})^{2}-X_{0}X_{3}) \subset \mathbb{P}^{4}.\]
Moreover, $S$ is singular along a line $\ell = V(X_{1},X_{3},X_{4})$.
Then we define a complete intersection 
\[\hat{X} := V(X_{1}X_{4}-X_{2}X_{3}+X_{0}X_{5},(X_{1}+X_{4})^{2}-X_{0}X_{3}+X_{5}^{2}) \subset \mathbb{P}^{5},\]
which is a Gorenstein terminal Fano threefold.
Indeed, $\hat{X}$ has a singularity of type $\mathit{cA}_{3}$ at $p = (0:0:1:0:0:0)$.
On the other hand, by construction, the hyperplane section $\hat{X} \cap V(X_{5})$ is isomorphic to $S$.
Let $V \subset |\mathcal{O}_{\hat{X}}(1)|$ be the pencil spanned by $X_{2}$ and $X_{5}$.
Blowing up $\hat{X}$ along the base locus of $V$, we get a Gorenstein terminal Fano threefold $X$.
Then $X$ admits a fibration $\pi \colon X \rightarrow \mathbb{P}^{1}$. 
Moreover, there is a fiber of $\pi$ isomorphic to $S$, as required.
\end{ex}
\begin{ex}\label{typeII}
Consider the morphism $\phi \colon \mathbb{P}^{2} \rightarrow \mathbb{P}^{3}$ defined by 
\[(x:y:z)\mapsto (xy:yz:zx:x^{2}+y^{2}+z^{2}).\]
The image $S$ is a quartic surface
\[S = V(X_{0}^{2}X_{1}^{2}+X_{1}^{2}X_{2}^{2}+X_{2}^{2}X_{0}^{2}-X_{0}X_{1}X_{2}X_{3}) \subset \mathbb{P}^{3}.\]
This surface is so-called a Steiner surface.
$S$ is singular along three lines $\ell_{0} = V(X_{1},X_{2}), \ell_{1} = V(X_{2},X_{0}), \ell_{2} = V(X_{0},X_{1})$.
Then we define a quartic threefold
\[X := V(X_{0}^{2}X_{1}^{2}+X_{1}^{2}X_{2}^{2}+X_{2}^{2}X_{0}^{2}-X_{0}X_{1}X_{2}X_{3}+X_{3}^{3}X_{4}+X_{3}X_{4}^{3}+X_{4}^{4}) \subset \mathbb{P}^{4},\]
which is a Gorenstein terminal Fano threefold.
Indeed, $X$ has singularities of type $\mathit{cA}_{3}$ at $p_{0} = (1:0:0:0:0), p_{1} = (0:1:0:0:0), p_{2} = (0:0:1:0:0)$.
On the other hand, by construction, the hyperplane section $X \cap V(X_{4})$ is isomorphic to $S$.
Therefore, $X$ contains a non-normal surface $S$ with $a(S,-K_{X}) = \frac{3}{2} > a(X,-K_{X})$.
\end{ex}

\begin{cor}\label{nondom}
Let $X$ be a Gorenstein terminal Fano threefold.
Let $M \subset \overline{\mathrm{Rat}}(X)$ be a non-dominant component parametrizing stable maps of anticanonical degree $d\le 2$ and let $Z$ be the image of the evaluation map $\mathrm{ev}\colon M'\rightarrow X$, where $M' \subset \overline{\mathrm{Rat}}_{1}(X)$ is the corresponding component.
\begin{enumerate}[(1)]
\item If $d = 1$ and the dimension of $M$ is greater than expected, that is, $\dim M > 1$, then $(Z, -K_{X}|_{Z})$ is either a contractible divisor of type E5 or a non-$\mathbb{Q}$-Cartier divisor isomorphic to $(\mathbb{P}^{2}, \mathcal{O}(1))$. 
\item If $d= 2$ and the dimension of $M$ is greater than expected, that is, $\dim M > 2$, then $Z$ is a surface listed in Theorem \ref{highA}, (1) and (2), except for a contractible divisor of type E2 and a non-$\mathbb{Q}$-Cartier divisor isomorphic to $(\mathbb{P}^{2}, \mathcal{O}(2))$.
\end{enumerate}
\end{cor}

\begin{proof}
    If the evaluation map is dominant, then $M$ has the expected dimension by Lemma \ref{free_iff_dominant}.
    Hence $Z$ is a $2$-dimensional subvariety of $X$ satisfying $a(Z, -K_{X}) > 1$ by Proposition \ref{NonDom-highA}.
    Then one can check the claims for each surface listed in Theorem \ref{highA}: for instance, if the normalization of $(Z, -K_{X}|_{Z})$ is isomorphic to $(\mathbb{P}^{2}, \mathcal{O}(2))$, then the space $M$ of $-K_{X}$-conics contained in $Z$ is birational to the Grassmannian $\mathbb{G}(1,2)$ of lines on $\mathbb{P}^{2}$. 
    This implies that contractible divisors of type E2, non-$\mathbb{Q}$-Cartier divisors isomorphic to $(\mathbb{P}^{2},\mathcal{O}(2))$, and surfaces listed in (3) of Theorem \ref{highA} are excluded from the case (2) of the claim. 
    The argument for (1) is similar. 
\end{proof}

\section{$a$-covers}
Let $X$ be a Gorenstein terminal Fano threefold and let $f\colon Y\rightarrow X$ be a dominant generically finite morphism from a smooth threefold $Y$ of degree at least $2$.
Note that for such a morphism, we have $a(Y, -f^{*}K_{X})\le a(X, -K_{X})$ by Lemma \ref{a-inv_dominant_finite}.
Recall that when $a(Y, -f^{*}K_{X}) = a(X, -K_{X})$, we say that $f$ is an $a$-cover .

\begin{prop}\label{factor through a-cover}
    Let $X$ be a Gorenstein terminal Fano threefold. 
    Let $M \subset \overline{\mathrm{Rat}}(X)$ be a component with the evaluation map $s\colon M'\rightarrow X$.
    Suppose that $s$ is dominant but its general fiber is not irreducible.
    Take a smooth resolution $\tilde{M}'\rightarrow M'$ and consider the composition $\tilde{s}\colon \tilde{M}'\rightarrow X$.
    Then the Stein factorization $f\colon Y\rightarrow X$ of $\tilde{s}$ is an $a$-cover.
\end{prop}

\begin{proof}
    One can run the argument as in \cite{Lehmann2019}*{Proposition 5.15}.
    Since $\tilde{s}$ factors through $f$, general stable maps $(g\colon \mathbb{P}^{1}\rightarrow X)\in M$ induce stable maps $(h\colon \mathbb{P}^{1}\rightarrow Y) \in \overline{\mathrm{Rat}}(Y)$. 
    Take a smooth resolution $\tilde{Y}\rightarrow Y$ and let $\tilde{N} \subset \overline{\mathrm{Rat}}(\tilde{Y})$ be the component containing  strict transforms $\tilde{h}$ of $h$.
    Since strict transforms $\tilde{h}$ dominate $\tilde{Y}$, we have $\tilde{N}\in \overline{\mathrm{Free}}(\tilde{Y})$ by Lemma \ref{free_iff_dominant}, and strict transforms $\tilde{h}$ are indeed general free maps parametrized by $\tilde{N}$.
    Hence $\dim \tilde{N} = -K_{\tilde{Y}}\cdot \tilde{h}_{*}\mathbb{P}^{1}$.
    On the other hand, $M \in \overline{\mathrm{Free}}(X)$ again by Lemma \ref{free_iff_dominant}. 
    Thus $\dim M = -K_{X}\cdot g_{*}\mathbb{P}^{1}$. 
    Since $\dim M = \dim \tilde{N}$, we obtain that $(K_{\tilde{Y}} - f^{*}K_{X})\cdot \tilde{h}_{*}\mathbb{P}^{1} = 0$. 
    This equality implies that $K_{\tilde{Y}} - f^{*}K_{X}$ lies on the boundary of $\overline{\mathrm{Eff}}_{1}(\tilde{Y})$ since $\tilde{h}_{\mathbb{P}^{1}}$ is a movable curve class, in particular, a nef curve class on $\tilde{Y}$. 
    Therefore, we conclude that $a(Y, -f^{*}K_{X}) = 1$. 
\end{proof}

We will study $a$-covers based on the Iitaka dimension $\kappa \in \{0,1,2\}$ of $K_{Y} - f^{*}K_{X}$.
\subsection{Iitaka dimension $2$ case}
\begin{lem}\label{conicbdl}
Let $X$ be a Gorenstein terminal Fano threefold and let $f\colon Y\rightarrow X$ be an $a$-cover with $Y$ smooth. 
Suppose that the Iitaka fibration $\psi \colon Y\rightarrow Z$ for $K_{Y}-f^{*}K_{X}$ is a morphism and $\dim Z = 2$.
Then there exist a component $M \subset \overline{\mathrm{Rat}}(X)$ parametrizing $-K_{X}$-conics and the corresponding component $M' \subset \overline{\mathrm{Rat}}_{1}(X)$ such that $\psi$ is birationally a base change of the forgetful morphism $\rho \colon M'\rightarrow M$ and $f$ factors through the evaluation map $\mathrm{ev}\colon M'\rightarrow X$ as rational maps.
\end{lem}
\begin{proof}
    Although the argument in the proof of \cite{Beheshti2022}*{Lemma 5.2} is valid for our case, we give a proof for completeness.
    Since $\dim Z = 2$, the general fiber of $\psi$ is a curve $D$ such that $a(D, -f_{*}K_{X}) = a(Y, -f^{*}K_{X}) = 1$.
    Then we have 
    \[\frac{\deg K_{D}}{f^{*}K_{X}\cdot D} = 1,\]
    which yields that $D$ is rational and $-K_{X}\cdot f_{*}D = -f^{*}K_{X}\cdot D = 2$.
    Since $-K_{X}$-lines cannot form a dominant family of rational curves on $X$, the restriction $f|_{D}\colon D\rightarrow C \subset X$ is birational onto a $-K_{X}$-conic $C$. 
    This correspondence defines a rational map $Z\dashrightarrow M$, where $M\subset \overline{\mathrm{Rat}}(X)$ is a component generically parametrizing birational stable maps onto $-K_{X}$-conics. 
    Let $U\subset Z$ be an open subscheme where any fiber of $\psi$ is a curve which maps birationally to a $-K_{X}$-conic.
    Since $M$ generically parametrizes birational stable maps, an open subscheme $V\subset M$ can embed to the Hilbert scheme $\mathrm{Hilb}(X)$ of $-K_{X}$-conics, and $\rho^{-1}(V) \subset M'$ embeds to the universal family $\mathrm{Univ}(X)$.
    By the universal property of the Hilbert scheme, we obtain a Cartesian diagram 
    \[
        \xymatrix{
            \psi^{-1}(U) \ar[r] \ar[d]_-{\psi} & \rho^{-1}(V) \ar[d]^-{\rho} \\
            U \ar[r] & V,
        }
    \]
    which yields the desired properties, completing the proof.
\end{proof}

\subsection{Iitaka dimension $1$ case}
Let $X$ be a Gorenstein terminal Fano threefold.
Assume furthermore that $X$ admits a del Pezzo fibration $\rho\colon X\rightarrow B$.
If we take any finite morphism $Z\rightarrow B$ of degree $\ge 2$, then the base change $f\colon X\times_{B} Z \rightarrow X$ is an $a$-cover.
The author was not able to see when the converse is true.
We only show the non-existence of $a$-covers with Iitaka dimension $1$ when $X$ is a del Pezzo threefolds with $-K_{X}$ is very ample.
See Definition \ref{Def_dP} for the definition of del Pezzo varieties.
\begin{lem}\label{adjoint_rigid_surface}
Let $X$ be a factorial del Pezzo threefold with an ample generator $H$.
Suppose that $H^{3} \ge 2$.
Then $X$ does not have any surface $S$ with $a(S,-K_{X}|_{S}) = 1$ and $\kappa(S, K_{S}-K_{X}|_{S})=0$.
\end{lem}
\begin{proof}
    Suppose the contrary that $X$ contains a surface $S$ with $a(S,-K_{X}|_{S}) = 1$ and $\kappa(S, K_{S}-K_{X}|_{S})=0$. 
    Let $\nu\colon \bar{S}\rightarrow S$ be the normalization. 
    Since $a (S, H) = a(X, H) = 2$ and $\kappa(S, K_{S}-K_{X}|_{S})=0$, $(\bar{S}, \nu^{*}H)$ is isomorphic to $(Q, \mathcal{O}_{Q}(1))$, where $Q$ is a quadric surface in $\mathbb{P}^{3}$. 
    Since $\dim |\mathcal{O}_{Q}(1)| = 3$, any strict sublinear system of $|\mathcal{O}_{Q}(1)|$ defines a map to $\mathbb{P}^{d}$ for $d \le 2$. 
    Thus $\nu$ must be an isomorphism and $(S, H) \cong (Q, \mathcal{O}_{Q}(1))$.
    Then we have $H^{2}\cdot S = 2$.
    Hence the only possibility is that $H^{3} = 2$ and $S \in |H|$.
    In this case, $H$ defines a double cover $\phi_{|H|}\colon X\rightarrow \mathbb{P}^{3}$ ramified along a quartic surface $B \subset \mathbb{P}^{3}$, and $S$ is the inverse image of a hyperplane of $\mathbb{P}^{3}$.
    Since $B$ does not contain any hyperplane, the restriction $\phi_{|H|}|_{S} = \phi_{|H|_{S}|}\colon S\rightarrow \mathbb{P}^{3}$ is also a ramified finite morphism.
    However, it contradicts the very ampleness of $\mathcal{O}_{Q}(1)$.
\end{proof}

As an immediate corollary, we obtain the following:
\begin{cor}\label{A-cover1_dP}
    Let $X$ be a factorial del Pezzo threefold with an ample generator $H$.
    Suppose that $H^{3} \ge 2$.
    Then there does not exist any $a$-cover $f\colon Y\rightarrow X$ with $\kappa(Y, K_{Y} - f^{*}K_{X}) = 1$.
\end{cor}
\begin{proof}
    Taking birational modifications $\psi\colon \tilde{Y}\rightarrow Y$ and $\phi\colon \tilde{X}\rightarrow X$, we can obtain a generically finite morphism $\tilde{f}\colon \tilde{Y}\rightarrow \tilde{X}$. 
    Since the $a$-invariant is a birational invariant (\cite{Hassett2015}*{Proposition 2.7}), we have 
    \[a(\tilde{Y}, -\tilde{f}^{*}\phi^{*}K_{X}) = a(Y, -f^{*}K_{X}) = a(X, -K_{X}) = a(\tilde{X}, -\phi^{*}K_{X}). \]
    Moreover, the Iitaka dimension of $K_{\tilde{Y}}-\tilde{f}^{*}\phi^{*}K_{X}$ is $1$ by assumption. 
    Then \cite{Lehmann2022}*{Lemma 4.9} shows that the image $\tilde{S}\subset \tilde{X}$ of a general fiber of the Iitaka fibration via $\tilde{f}$ is a surface with $a(\tilde{S}, -\phi^{*}K_{X}) = a(\tilde{X},-\phi^{*}K_{X})$ and $\kappa(\tilde{S}, K_{\tilde{S}}-\phi^{*}K_{\tilde{X}})=0$. 
    Since $\tilde{S}$ is birational to the image $S:=\phi(S)$, $S$ satisfies $a(S, -K_{X}|_{S}) = a(X, -K_{X})$ and $\kappa(S, K_{S}-K_{X}|_{S})=0$. 
    However, it contradicts Lemma \ref{adjoint_rigid_surface}. 
\end{proof}

\subsection{Iitaka dimension $0$ case}
We say that an $a$-cover $f\colon Y\rightarrow X$ is adjoint rigid if $\kappa(Y, K_{Y}-f^{*}K_{X}) = 0$.
\begin{thm}[cf. \cite{Beheshti2022}*{Theorem 5.8}]\label{AdjointRigid}
Let $X$ be a terminal factorial Fano threefold without E5 contractions.
Then there exist no adjoint rigid $a$-covers.
\end{thm}
\begin{proof}
Let $f\colon Y \rightarrow X$ be an adjoint rigid $a$-cover with $Y$ smooth.
As a consequence of relative $K_{Y}$-MMP, we obtain a birational map $\phi\colon Y\dashrightarrow \bar{Y}$ to a $\mathbb{Q}$-factorial terminal Fano threefold, and a generically finite morphism $\bar{f}\colon \bar{Y}\rightarrow X$ such that $f = \bar{f} \circ \phi$.
This process can be seen as steps of $(K_{Y}-f^{*}K_{X})$-MMP.
Hence we have $\kappa(\bar{Y}, K_{\bar{Y}}-\bar{f}^{*}K_{X}) = \kappa(Y, K_{Y}-f^{*}K_{X}) = 0$.
Now we can apply the argument in \cite{Beheshti2022}*{\S 5.3} where they use the classification of divisorial contractions of $\mathbb{Q}$-factorial terminal threefolds. 
Then we see that the image $S$ of the adjoint divisor $K_{\bar{Y}}-\bar{f}^{*}K_{X}$ is swept out by a $2$-dimensional family of $-K_{X}$-lines.
Since we assume $X$ is factorial, $S$ is an E5 divisor by Corollary \ref{nondom}.
However, it contradicts the assumption.
Thus, the ramification divisor $K_{\bar{Y}}-\bar{f}^{*}K_{X}$ of $\bar{f}$ is equal to zero.
Since $X$ is a terminal factorial Fano threefold, we see that the ramification locus of any finite morphism has pure codimension $1$, and that $X$ is simply connected by \cite{Hacon2007}. 
It yields the contradiction that $f$ is birational.
Thus the claim holds.
\end{proof}

Combining Lemma \ref{conicbdl}, Corollary \ref{A-cover1_dP}, and Theorem \ref{AdjointRigid}, we obtain Theorem \ref{a-covers}.

\section{Movable Bend-and-Break}
In this section, we prove Movable Bend-and-Break (MBB for short).
Since the existence of terminal factorial singularities has little effect on the argument in \cite{Beheshti2022}*{\S.6}, one can obtain very similar results.
\par
Recall that $\overline{\mathrm{Rat}}(X) \subset \overline{M}_{0,0}(X)$ denotes the union of irreducible components generically parametrizing stable maps with irreducible domains, and $\overline{\mathrm{Free}}(X) \subset \overline{\mathrm{Rat}}(X)$ denotes the union of irreducible components generically parametrizing free stable maps.
\subsection{MBB for multiple covers}
For a free curve $f\colon \mathbb{P}^{1}\rightarrow X$, write the restricted tangent bundle $f^{*}T_{X}$ as 
\[f^{*}T_{X} = \mathcal{O}(a_{1})\oplus \mathcal{O}(a_{2})\oplus \mathcal{O}(a_{3}),\]
with $a_{1}\ge a_{2}\ge a_{3} \ge 0$.
If $a_{2}\ge 1$, then a general deformation of $f$ is an immersion by \cite{Kollar1996}*{II.3.14}.
When $a_{2}= a_{3} = 0$, the argument is similar to the case when $X$ is a smooth Fano threefold.
Note that if $X$ is Gorenstein terminal Fano threefold, any dominant component of $\overline{\mathrm{Rat}}(X)$ generically parametrizes free curves by Lemma \ref{free_iff_dominant}.
\begin{prop}\label{multiple_covers}
Let $X$ be a Gorenstein terminal Fano threefold and let $M \subset \overline{\mathrm{Free}}(X)$ be a component generically parametrizing stable maps $f\colon \mathbb{P}^{1}\rightarrow X$ such that $f^{*}T_{X} = \mathcal{O}(a)\oplus \mathcal{O}^{\oplus 2}$, where $a\ge 3$.
Then the general member of $M$ is a multiple cover of a $-K_{X}$-conic. 
\end{prop}
\begin{proof}
    We run the argument similar to the proof of \cite{Kollar1996}*{II.3.14.4}.
    Let $(f\colon \mathbb{P}^{1}\rightarrow X)\in M$ be a free stable map such that $f^{*}T_{X} = \mathcal{O}(a) \oplus \mathcal{O}^{\oplus 2}$.
    Since $f$ is free, the image of $f$ is disjoint from the singular locus of $X$. 
    By L\"{u}roth's theorem, $f$ is the composition of a finite morphism $g\colon \mathbb{P}^{1} \rightarrow \mathbb{P}^{1}$ of degree $r$ and a birational morphism $h\colon \mathbb{P}^{1}\rightarrow X$.
    Let $N \subset \overline{\mathrm{Free}}(X)$ be the irreducible component parametrizing deformations of $h$.
    Then $\dim N = \deg h^{*}T_{X}$ since $h$ is also free.
    Since $\dim M = a = r\deg h^{*}T_{X}$, and on the other hand, $\dim M = 2r -2 + \deg h^{*}T_{X}$, we see that $r = 1$ or $\deg h^{*}T_{X} = 2$.
    \par
    Suppose that $\deg h^{*}T_{X} \ge 3$.
    Take a point $p \in \mathbb{P}^{1}$ and we may assume that $h$ maps $p$ to a general point $x$ of $X$.
    Since $\dim N = \deg h^{*}T_{X} \ge 3$, we see that the subset $Z \subset X$ dominated by rational curves in $N$ passing through $x$ has dimension $\ge 2$.
    Then, by \cite{Kollar1996}*{II.3.10}, $f^{*}T_{X}$ must have at least two positive summands, a contradiction.
    Thus we conclude that $\deg h^{*}T_{X} = 2$ and $r\ge 2$.
\end{proof}

\begin{cor}[cf. \cite{Beheshti2022}*{Proposition 6.2}]\label{MBB_multiple_covers}
Let $X$ and $M$ be as in Proposition \ref{multiple_covers}.
Then $M$ contains a stable map which is a union of two free curves.
\end{cor}
\begin{proof}
By Proposition \ref{multiple_covers}, the general element $f\colon \mathbb{P}^{1}\rightarrow X$ of $M$ can be seen as the composition of an $r$-sheeted cover $g\colon \mathbb{P}^{1} \rightarrow \mathbb{P}^{1}$ and a birational morphism $h\colon \mathbb{P}^{1}\rightarrow X$ to a $-K_{X}$-conic.
Suppose $h$ is a general map in the component of $\overline{M}_{0,0}(X)$ parametrizing deformations of $h$, and let $h^{*}T_{X} = \mathcal{O}(a_{1}) \oplus \mathcal{O}(a_{2}) \oplus \mathcal{O}(a_{3})$ such that $a_{1} + a_{2} + a_{3} = 2$. 
Then we have $f^{*}T_{X} = \mathcal{O}(a_{1}r) \oplus \mathcal{O}(a_{2}r) \oplus \mathcal{O}(a_{3}r)$. 
Since we assume $f$ is a free stable map, we see that $h^{*}T_{X} = \mathcal{O}(2) \oplus \mathcal{O}^{\oplus 2}$. 
Now one can deform $g$ into a stable map $g'\colon C_{1}\cup C_{2}\rightarrow \mathbb{P}^{1}$, which is a union of an $r_{1}$-sheeted cover $g'|_{C_{1}}\colon C_{1}\rightarrow \mathbb{P}^{1}$ and an $r_{2}$-sheeted cover $g|_{C_{2}}\colon C_{2}\rightarrow \mathbb{P}^{1}$, where $r_{1} + r_{2} = r$. 
Then $f$ deforms to a stable map $f' := h\circ g'\colon C_{1}\cup C_{2}\rightarrow X$, which is a union of free stable maps since $(f'|_{C_{i}})^{*}T_{X} = \mathcal{O}(2r_{i}) \oplus \mathcal{O}^{\oplus 2}$ for $i\in \{1,2\}$. 
\end{proof}
\subsection{MBB for birational stable maps}
We recall the notion of basepoint free families from \cite{Beheshti2022}: 
\begin{dfn}[\cite{Beheshti2022}*{Definition 6.3}]
Let $X$ be a normal projective variety.
A basepoint free family of reduced irreducible curves on $X$ consists of a diagram 
\[
\xymatrix{
    \mathcal{U} \ar[r]^-{s} \ar[d]_-{p} & X, \\
    V & 
}
\]
where $V$ is irreducible, $p$ is a proper flat morphism with reduced irreducible fibers of dimension $1$, and $s$ is a flat morphism which maps any fiber of $p$ birationally to a curve on $X$.
\end{dfn}
In this paper, we only consider basepoint free families constructed as follows:
\begin{constr}
    Let $X$ be a normal projective variety of dimension $n$ and let $L$ be a big and basepoint free divisor on $X$.
    Consider the Grassmannian $\mathbb{G} := \mathbb{G}(n-2, |L|)$ of $(n-2)$-dimensional sublinear systems of $|L|$.
    By Bertini theorem, there is an open dense subset $V\subset \mathbb{G}$ such that for any $\mathfrak{d}\in V$, the base locus $\mathrm{Bs}(\mathfrak{d})$ is a reduced irreducible curve on $X$.
    We set 
    \[\mathcal{U} = \{(x, \mathfrak{d}) \mid x \in \mathrm{Bs}(\mathfrak{d}))\}\subset X\times V. \]
    Then the diagram 
    \[
    \xymatrix{
        \mathcal{U} \ar[r]^-{s} \ar[d]_-{p} & X, \\
        V & 
    }
    \]
    is a basepoint free family of reduced irreducible curves on $X$, where $p$ and $s$ are the projections to each factor.
    We say that the basepoint free family $p$ is constructed from $L$.
\end{constr}

Now we prove one of the main theorems: 
\newtheorem*{MainThm2}{\rm\bf Theorem \ref{MBB_noE5}}
\begin{MainThm2}[Movable Bend-and-Break]
    Let $X$ be a terminal factorial Fano threefold without E5 contractions.
    Then any component $M$ of $\overline{\mathrm{Free}}(X)$ parametrizing stable maps of anticanonical degree at least $5$ contains a stable map $g\colon C\rightarrow X$ such that 
    \begin{itemize}
        \item the domain $C$ consists of two irreducible components $C_{1}, C_{2}$, and 
        \item each restriction $g|_{C_{i}}\colon C_{i}\rightarrow X$ is a free rational curve.
    \end{itemize}
\end{MainThm2}

\begin{proof}
The argument is similar to \cite{Beheshti2022}*{Proposition 6.5 and Theorem 6.7}. 
By Corollary \ref{MBB_multiple_covers}, it suffices to consider the case that the restricted tangent bundle $f^{*}T_{X}$ has at least two positive summands for a general $[f] \in M$.

Take a smooth resolution $\phi \colon \tilde{X}\rightarrow X$. 
Let $\tilde{M}$ be the irreducible component containing strict transforms $\tilde{f}\colon \mathbb{P}^{1}\rightarrow \tilde{X}$ of free stable maps $f\colon \mathbb{P}^{1}\rightarrow X$ parametrized by $M$. 
Since the strict transforms $\tilde{f}$ avoid the exceptional locus of $\phi$, these are also free stable maps, and are general members of $\tilde{M}$.  
Moreover, the normal bundle has the same form $N_{\tilde{f}/\tilde{X}} \cong N_{f/X} \cong \mathcal{O}(a) \oplus \mathcal{O}(b)$, where $b \ge a \ge 0$ and $a + b \ge 3$. 

Now we take the contraction $X\rightarrow X'$ that contracts all divisors of type E3, E4, and E1 of the form $\mathbb{P}^{1}\times \mathbb{P}^{1}$. 
We construct the basepoint free family $p\colon \mathcal{U}\rightarrow V$ on $\tilde{X}$ from the pull back of a very ample divisors on $X'$ by the composition $\phi'\colon \tilde{X}\rightarrow X\rightarrow X'$. 
By \cite{Beheshti2022}*{Lemma 6.4}, we obtain the sublocus $\tilde{T} \subset \tilde{M}$ which is the closure of one parameter family of stable maps $\tilde{f}\colon \mathbb{P}^{1} \rightarrow \tilde{X}$ whose images contain general $r$ points of $\tilde{X}$ and meet with general $s$ members of the basepoint free family $p$, where 
\[
(r, s)=
\begin{cases}
    (a, 1), & \mbox{if } a = b\\
    (a+1, b-a-1), & \mbox{if } a < b.
\end{cases}
\]
Let $(\tilde{g}\colon C\rightarrow \tilde{X})\in \tilde{T}$ be a stable map with reducible domain. 
Write 
\[C = \bigcup_{i=1}^{l}C_{i} \cup \bigcup_{j=1}^{m}D_{j} \cup \bigcup_{k=1}^{n}E_{k} \cup (\mbox{contracted components}), \]
where $\tilde{g}|_{C_{i}}$ are free, $\tilde{g}|_{D_{j}}$ are non-free with $\tilde{g}(D_{j})\not\subset \mathrm{Exc}(\phi)$, and $\tilde{g}|_{E_{k}}$ are non-free with $\tilde{g}(E_{k})\subset \mathrm{Exc}(\phi)$. 
Let $(g\colon \bar{C}\rightarrow X) \in M$ be the stabilization of the composition $\phi \circ \tilde{g}$, where 
\[\bar{C} = \bigcup_{i=1}^{l}C_{i} \cup \bigcup_{j=1}^{m}D_{j} \cup (\mbox{contracted components}). \]
Note that each $\tilde{g}|_{C_{i}}$ cannot meet with the exceptional locus $\mathrm{Exc}(\phi)$: if deformations of a rational curve that meet with $\mathrm{Exc}(\phi)$ dominate $\tilde{X}$, then their images on $X$ pass through the singular locus of $X$ and dominate $X$, which contradicts Lemma \ref{free_iff_dominant}.

Since members of $\tilde{T}$ contain at least two general points, bend-and-break (\cite{Lehmann2023}*{Lemma 4.1}) shows that $l\ge 2$. 
Thus it is enough to show that $m = n = 0$. 
Indeed, if we find a union of $l\ge 3$ free curves on $\tilde{X}$, then by smoothing $l-1$ components of them, we obtain a union of two free curves on $\tilde{X}$. 
Moreover, since free curves are disjoint from the exceptional locus of $\phi$, the image by $\phi$ is also a union of two free curves on $X$, as required. 
We prove the claim dividing into two cases. 
\begin{enumerate}[(I)]
    \item Suppose $a=b$. 
    Since any non-free rational curve of anticanonical degree at most $a+b+2$ is contained in a proper closed subset by Theorem \ref{Kollar dom iff free}, general points $p_{1}, \dots, p_{a}$ can only meet with free curves $\tilde{g}(C_{i})$. 
    Suppose the normal sheaf of the general deformation of $\tilde{g}|_{C_{i}}$ has the form $\mathcal{O}(a_{i}) \oplus \mathcal{O}(b_{i}) \oplus (\mbox{torsion part})$, where $b_{i}\ge a_{i}\ge 0$. 
    Then each $\tilde{g}(C_{i})$ contains at most $a_{i}+1$ general points. 
    Hence we obtain that 
    \begin{align*}
        a \le \sum_{i=1}^{l}(a_{i}+1)
         &\le \frac{1}{2}\sum_{i=1}^{l}(a_{i}+b_{i}+2)\\
         &= \frac{1}{2}\sum_{i=1}^{l}(-K_{\tilde{X}}\cdot \tilde{g}_{*}C_{i})\\
         &= \frac{1}{2}\sum_{i=1}^{l}(-K_{X}\cdot g_{*}C_{i})\\
         &\le \frac{1}{2}(-K_{X}\cdot g_{*}\bar{C})
         = \frac{1}{2}(a+b+2)
         = a+1. 
    \end{align*}
    In particular, we see that $m\le 2$ and $\sum_{j=1}^{m}(-K_{X}\cdot g_{*}D_{j}) \le 2$. 
    
    Suppose that $m+n>0$. 
    If $n>0$, then $m>0$ since any free curve cannot meet the exceptional locus. 
    Hence, in any case, we find a non-free component $D_{j}$. 
    
    We first assume that the general member $Q$ of the basepoint free family $p$ meets with $\tilde{g}(C_{i})$ for some $1\le i\le l$, say $i=1$. 
    Then by the inequality, we see that $m=1$, $-K_{X}\cdot g_{*}D_{1} = 1$, and general points $p_{1},\dots, p_{a}$ and the curve $Q$ determine finitely many possibilities of $\tilde{g}|_{C_{i}}$. 
    Hence all $\tilde{g}(C_{i})$ are disjoint. 
    Since $l\ge 2$, the $-K_{X}$-line $g(D_{1})$ must meet with $g(C_{i})$ for all $1\le i\le l$. 
    Then the space of deformations of $g(D_{1})$  must have dimension at least $l> 1$. 
    By Corollary \ref{nondom}.(1), $g(D_{1})$ deforms in an E5 divisor, which contradicts the assumption. 

    We next assume that the general member $Q$ of the basepoint family $p$ is disjoint from all $\tilde{g}(C_{i})$. 
    Then $Q$ must meet with $\tilde{g}(D_{j})$ for some $1\le j \le m$, say $j=1$. 
    Note that $\bigcup_{j=1}^{m}\tilde{g}(D_{j})$ must meet $Q$ and all $\tilde{g}(C_{i})$. 
    \begin{enumerate}[(i)]
        \item Assume that $m=1$ and $-K_{X}\cdot g_{*}D_{1} = 1$. 
        Then at most one $\tilde{g}(C_{i})$ can deform in dimension $1$ fixing general $a_{i}+1$ points, and any other $\tilde{g}(C_{i'})$ cannot deform fixing general $a_{i'}+1$ points. 
        Hence the space of deformations of $g(D_{1})$ must have dimension at least $1+(l-1) = l \ge 2$, which contradicts the absence of E5 divisors. 

        \item Assume that $\sum_{j=1}^{m}(-K_{X}\cdot g_{*}D_{j}) = 2$. 
        Then any $\tilde{g}(C_{i})$ cannot deform fixing general $a_{i}+1$ points.
        Then, $g(D_{1})$ is either a $-K_{X}$-line that meets with $\phi(Q)$ and at least one $g(C_{i})$, or a $-K_{X}$-conic that meets with $\phi(Q)$ and all $g(C_{i})$. 
        The former case contradicts the absence of E5 divisors again. 
        For the latter case, $g(D_{1})$ must deform in a surface listed in Corollary \ref{nondom}.(2), i.e., contractible divisors of type E3, E4, or E1 of the form $\mathbb{P}^{1}\times \mathbb{P}^{1}$. 
        However, $\phi(Q)$ cannot meet such surfaces since the basepoint free family $p$ is constructed from the pull back of a very ample divisor on $X'$. 
        Hence this case is also impossible. 
    \end{enumerate}
    Therefore, we conclude the claim when $a=b$. 

    \item Suppose $b\ge a+1$. 
    Although the argument is very similar to \cite{Beheshti2022}*{Proposition 6.5}, we give a proof for completeness. 

    Let $\tilde{Y} \subset \tilde{X}$ be the $2$-dimensional subset containing all non-free curves of anticanonical degree at most $a+b+2$. 
    Let $\tilde{\Sigma} \subset \tilde{X}$ be the surface swept out by stable maps that pass through the general $a+1$ points. 
    Since $\tilde{\Sigma}$ contains general points, we see that the intersection $\tilde{Y}\cap \tilde{\Sigma}$ is a proper closed set of $\tilde{\Sigma}$. 
    Hence general members $Q_{1}, \dots, Q_{b-a-1}$ of the basepoint free family $p$ meet with $\tilde{\Sigma}$ outside of $\tilde{Y}\cap \tilde{\Sigma}$. 
    Since $\tilde{g}(D_{j})$ and $\tilde{g}(E_{k})$ are contained in $\tilde{Y}\cap \tilde{\Sigma}$, we see that $Q_{1}, \dots, Q_{b-a-1}$ can only meet with free curves $\tilde{g}(C_{i})$. 
    Suppose each $\tilde{g}(C_{i})$ meets with $\alpha_{i}$ general points and $\beta_{i}$ general members of the basepoint free family $p$. 
    Then by \cite{Beheshti2022}*{Lemma 6.4}, we obtain that 
    \begin{align*}
        a + b + 1 \le \sum_{i=1}^{l}(2\alpha_{i} + \beta_{i})
         &\le \sum_{i=1}^{l}(-K_{\tilde{X}}\cdot \tilde{g}_{*}C_{i})\\
         &= \sum_{i=1}^{l}(-K_{X}\cdot g_{*}C_{i})\\
         &\le -K_{X}\cdot g_{*}\bar{C} = a + b + 2. 
    \end{align*}
    If $m+n>0$, then $m=1$ and $g(D_{1})$ is a $-K_{X}$-line. 
    In this case, general points $p_{1},\dots, p_{a}$ and general members $Q_{1},\dots, Q_{b-a-1}$ of the basepoint free family $p$ determine finitely many possibilities of $\tilde{g}|_{C_{i}}$. 
    Thus, all $\tilde{g}(C_{i})$ are disjoint. 
    Then as discussed above, the $-K_{X}$-line $g(D_{1})$ must deform in dimension $l>1$ to meet with all $g(C_{i})$, which contradicts the assumption. 
    Therefore, we have $m=n=0$, completing the proof. 
\end{enumerate}
\end{proof}

\section{Geometric Manin's Conjecture}
In this section, we prove Geometric Manin's Conjecture for a couple of examples.
\subsection{Statement of Geometric Manin's Conjecture}
Firstly, we will formulate Geometric Manin's Conjecture.
\begin{dfn}\label{breaking}
Let $X$ be a $\mathbb{Q}$-factorial terminal uniruled projective variety and let $L$ be a nef and big divisor on $X$.
Let $f\colon Y\rightarrow X$ be a generically finite morphism onto its image. 
\begin{enumerate}[(1)]
\item We say that $f$ is breaking with respect to $L$ if one of the following holds:
\begin{enumerate}[(i)]
\item $a(Y, f^{*}L) > a(X, L)$, 
\item $f$ is an $a$-cover with $\kappa(Y, K_{Y}-a(Y, f^{*}L)f^{*}L) > 0$, or 
\item $f$ is face contracting, i.e., $f$ is an $a$-cover such that the induced map $f_{*}\colon F(Y, f^{*}L)\rightarrow F(X, L)$ is not injective.
\end{enumerate}
\item A component $M$ of $\overline{\mathrm{Rat}}(X)$ is accumulating if there is a breaking morphism $f\colon Y\rightarrow X$ with respect to $-K_{X}$ and a component $N \subset \overline{\mathrm{Rat}}(Y)$ such that $f$ induces a dominant map $f_{*}\colon N\dashrightarrow M$.
\item If $M\subset \overline{\mathrm{Rat}}(X)$ is not accumulating, then $M$ is called a Manin component.
\end{enumerate}
\end{dfn}
\begin{conj}[Geometric Manin's Conjecture, \cite{Tanimoto2021}]\label{GMC}
Let $X$ be a $\mathbb{Q}$-factorial terminal weak Fano variety.
Then there exist a constant $c\in \mathbb{Z}$ and a nef curve class $\alpha \in \mathrm{Nef}_{1}(X)$ such that for any curve class $\beta \in \alpha + \mathrm{Nef}_{1}(X)$, there are exactly $c$ Manin components in $\overline{\mathrm{Rat}}(X, \beta)$.
\end{conj}

\begin{rem}\label{GMC_known}
    Geometric Manin's Conjecture is known for the following cases: 
    homogeneous varieties (\cite{Kim2001}, \cite{Thomsen1998}), 
    toric varieties (\cite{Bourqui2012}, \cite{Bourqui2016}), 
    smooth hypersurfaces of low degree (\cite{Harris2004}, \cite{Coskun2009}, \cite{Beheshti2013}, \cite{Browning2017}, \cite{Riedl2019}), 
    del Pezzo surfaces (\cite{Testa2005}, \cite{Testa2009}, \cite{Beheshti2023}), 
    smooth Fano threefolds (\cite{Lehmann2019}, \cite{Lehmann2021a}, \cite{Shimizu2022}, \cite{Beheshti2022}, \cite{Burke2022}), 
    smooth del Pezzo varieties (\cite{Coskun2009}, \cite{Lehmann2019}, \cite{Shimizu2022}, \cite{Okamura2022a}), and 
    del Pezzo fibrations (\cite{Lehmann2022a}, \cite{Lehmann2023}).
    \par
    The methods used in the first two cases (homogeneous varieties and toric varieties) are different from ours.
    We mainly follow the methods developed in the papers such as \cite{Harris2004}, \cite{Coskun2009}, \cite{Lehmann2019}, \cite{Beheshti2022}.
\end{rem}
%
%
%
\subsection{Geometric Manin's Conjecture for del Pezzo threefolds}
In this subsection, we will prove Geometric Manin's Conjecture for cubic and quartic del Pezzo threefolds with terminal factorial singularities.
We recall the definition of del Pezzo varieties.
\begin{dfn}[\cite{Kuznetsov2023}]\label{Def_dP}
    We say that a projective variety $X$ is del Pezzo if $X$ is terminal and there exists an ample Cartier divisor $H$ such that $-K_{X} = (\dim X - 1)H$.
    Such a divisor $H$ is called the fundamental divisor on $X$.
\end{dfn}
Cubic hypersurfaces and complete intersections of two quadrics with terminal singularities are examples of del Pezzo varieties.
See \cite{Kuznetsov2023} for a detailed classification of del Pezzo varieties.
\begin{rem}
In \cite{Iskovskikh1999}*{Definition 3.2.1}, they have defined del Pezzo varieties without the condition on terminal singularities.
Since in this paper, we consider Gorenstein terminal Fano threefolds, we assume that any del Pezzo variety has at worst terminal singularities.
\end{rem}
We now consider the spaces of low degree curves on a cubic threefold.
%
\begin{thm}\label{Lines_Cubic}
Let $X$ be a cubic threefold with terminal factorial singularities and let $H$ be the fundamental divisor on $X$.
Then the space $\overline{M}_{0,0}(X, 1)$ parametrizing $H$-lines is irreducible of dimension $2$.
\end{thm}
\begin{proof}
When $X$ is smooth, then Theorem \ref{Lines_Cubic} is proved in \cite{AltmanKleiman1977Fanoscheme}. 
Hence we consider the singular case. 

Choose a singular point $q\in X$ and we may assume that $q = (1:0:0:0:0)$.
Then $X$ is written as $X = V(x_{0}F + G)$, where $F\in k[x_{1},\dots, x_{4}]_{2}$ and $G\in k[x_{1},\dots, x_{4}]_{3}$.
Let $S$ be the surface swept out by lines through $q$, which is a cone of degree $6$ with the vertex $q$, in particular, $S = V(F)\cap X \in |2H|$.
Consider the image $C = \pi(S)$, where $\pi\colon \mathbb{P}^{4}\dashrightarrow \mathbb{P}^{3}$ is the projection from $q$.
Then $C = V(F, G)$ is a curve of degree $6$ and genus $4$.
Now, we consider the rank of the quadric form $F$. 
Since $X$ is factorial with isolated cDV singularities, the rank of $F$ is at least $2$ by \cite{MarquandViktorova2023defect}*{Proposition 3.4}. 
Then by \cite{MarquandViktorova2023defect}*{Theorem 1.1}, $C$ is irreducible when the rank of $F$ is greater than $2$, and $C$ has two irreducible components when the rank of $F$ is $2$. 
We then show the irreducibility of $\overline{M}_{0,0}(X, 1)$.
It is clear that $\overline{M}_{0,0}(X, 1)$ is isomorphic to the Fano surface $F(X)$ of lines.
We study the irreducibility of $F(X)$ depending on the rank of $F$. 
\begin{enumerate}[(1)]
\item Suppose that the rank of $F$ is greater than $2$, i.e., $C$ is irreducible. 
We claim that there is a birational map $\psi \colon \mathrm{Sym}^{2}(C) \dashrightarrow F(X)$.
Since $C$ is irreducible, so is $\mathrm{Sym}^{2}(C)$.
Any element of $\mathrm{Sym}^{2}(C)$ except for elements corresponding to singularities of $C$ defines a line $L$ on $\mathbb{P}^{3}$ i.e., for $(p_{1}, p_{2}) \in \mathrm{Sym}^{2}(C)$, if $p_{1}\neq p_{2}$, then $L$ is the line joining $p_{1}$ and $p_{2}$, and if $p_{1} = p_{2}$, then $L$ is the tangent line to $p_{1}$.
Let $P$ be the plane in $\mathbb{P}^{4}$ spanned by $L$ and $q$.
Then the intersection $X\cap P$ consists of three lines, two of which are the lines passing through $q$ and $p_{i}$ ($i\in \{1,2\}$).
So we define $\psi(p_{1}, p_{2})$ to be the residual line.
The inverse of $\psi$ is constructed as follows: 
 take a line $\ell \in F(X)$ not passing through $q$.
Since the cone $S = V(F, G)$ is a divisor in $|2H|$, where $H$ is the hyperplane class on $X$, hence $S\cdot \ell = 2$.
Then we see that $\psi^{-1}(\ell)$ is an element of $\mathrm{Sym}^{2}(C)$ corresponding to the image $\pi(S\cap \ell)$.
Therefore $\psi$ is birational and hence $F(X)$ is irreducible.

\item Suppose that the rank of $F$ is $2$, i.e., $C$ is a union of two irreducible curves $C_{1}, C_{2}$. 
Since $V(F) \subset \mathbb{P}^{3}$ consists of two distinct planes $P_{1}, P_{2}$, each $C_{i}$ is a plane cubic curve contained in $P_{i}$. 
Even in this case, one can define a map $\psi\colon \mathrm{Sym}^{2}(C)\dashrightarrow F(X)$ as in (1). 
For $i=1,2$, consider the composition $\psi_{i}\colon \mathrm{Sym}^{2}(C_{i})\hookrightarrow \mathrm{Sym}^{2}(C)\dashrightarrow F(X)$. 
For $(p_{1},p_{2})\in \mathrm{Sym}^{2}(C_{i})$, the line $L$ joining $p_{1}$ and $p_{2}$ is contained in the plane $P_{i}$. 
Since $C_{i}\subset P_{i}$ is a plane cubic curve, we have another point $p_{3}$ in $C_{i}\cap L$ other than $p_{1}, p_{2}$. 
Then we have $\psi_{i}\colon \mathrm{Sym}^{2}(C_{i})\ni(p_{1},p_{2}) \mapsto p_{3}\in C_{i}\subset F(X)$ regarding $C$ as the set of lines through $q$. 
Lastly, we consider the composition $\phi\colon C_{1}\times C_{2}\hookrightarrow C\times C\twoheadrightarrow \mathrm{Sym}^{2}(C) \dashrightarrow F(X)$. 
Combining with the fact that $\psi_{i}$ cannot dominate any component of $F(X)$, one can prove that $\phi$ is a birational map as discussed in (1). 
Since $C_{1}\times C_{2}$ is irreducible, so is $F(X)$, completing the proof. 
\end{enumerate}
\end{proof}
\begin{thm}\label{Conics_Cubic}
Let $X$ be a cubic threefold with terminal factorial singularities and let $H$ be the fundamental divisor on $X$.
Then the space $\overline{M}_{0,0}(X, 2)$ parametrizing $H$-conics has two irreducible components.
One component generically parametrizes free birational stable maps, and the other component parametrizes double covers of $H$-lines.
Both components have the expected dimension $4$.
\end{thm}
\begin{proof}
Let $N \subset \overline{M}_{0,0}(X, 2)$ be the locus parametrizing double covers of $H$-lines.
Since the space of double covers of $\mathbb{P}^{1}$ is irreducible of dimension $2$, hence $N$ is an irreducible component of $\overline{M}_{0,0}(X, 2)$ of dimension $4$ by Theorem \ref{Lines_Cubic}.
We next consider the incidence correspondence 
\[\Sigma := \{(\ell, P)\mid \ell \in F(X), P\in \mathbb{G}(2, 4), \ell \subset P\}.\]
Then the first projection $p\colon \Sigma \rightarrow F(X)$ is a $\mathbb{P}^{2}$-bundle, hence $\Sigma$ is irreducible of dimension $4$ by Theorem \ref{Lines_Cubic}.
Note that $X$ does not contain planes since $X$ is factorial.
Let $(\ell, P)\in \Sigma$, then there is a conic $\alpha$ such that $X\cap P = \ell \cup \alpha$.
This defines a morphism $\Sigma \rightarrow \overline{M}_{0,0}(X, 2)$.
Moreover, this morphism induces a birational morphism from $\Sigma$ to $\overline{M}_{0,0}(X, 2)$ away from the component $N$ since any conic except for double lines is contained in a unique plane.
Hence the claim holds.
\end{proof}
Next, we consider the spaces of low degree curves on a quartic del Pezzo threefold, that is, a complete intersection of two quadrics in $\mathbb{P}^{5}$.
\begin{thm}\label{Lines_Quartic}
Let $X$ be a quartic del Pezzo threefold with terminal factorial singularities and let $H$ be the fundamental divisor on $X$.
Then the space $\overline{M}_{0,0}(X, 1)$ parametrizing $H$-lines is irreducible of dimension $2$.
\end{thm}
\begin{proof}
When $X$ is smooth, then Theorem \ref{Lines_Quartic} is proved in \cite{Reid1972CIquadrics}. 
Hence we consider the singular case. 

We fix a singular point $q \in X$ and assume $q= (1:0:\dots :0)$.
Then we may assume that $X\subset \mathbb{P}^{5}$ is the intersection of a quadric cone $Q_{1}$ with vertex $q$ and a quadric hypersurface $Q_{2}$ which is smooth at $q$.
Let $\pi\colon \mathbb{P}^{5}\dashrightarrow \mathbb{P}^{4}$ be the projection from $q$. 
Write $Q' = \pi(Q_{1})$, which is a nonsingular quadric hypersurface in $\mathbb{P}^{4}$.
As in the proof of Theorem \ref{Lines_Cubic}, we let $S$ be the surface swept out by lines through $q$, and set $C = \pi(S)$.
Then $S = T_{q}Q_{2} \cap X$ is the intersection of $X$ and the tangent hyperplane of $Q_{2}$ to $q$. 
Hence we see that $S \in |H|$ is irreducible, and so is $C$.
\par
We consider the Fano scheme $F(Q')$ of lines on $Q'$, which is irreducible of dimension $3$.
Moreover, it is isomorphic to the Fano scheme $F_{2}(X)$ of planes on $Q_{1}$.
Let $\mathcal{U}\rightarrow F(Q')$ denote the universal family and let $s\colon \mathcal{U}\rightarrow Q'$ be the evaluation map.
Then $s$ is a smooth conic fibration.
Hence $\mathcal{U}_{C} = s^{-1}(C)$ is an irreducible surface since $C$ is irreducible.
We will construct a birational morphism $\rho \colon \mathcal{U}_{C}\rightarrow F(X)$.
Take an element $(L, c)\in \mathcal{U}_{C}$.
We note that $L$ corresponds to the plane $\pi^{-1}(L)$, and that $c$ corresponds to the line $\pi^{-1}(c)$.
By the choice of $L$, $\pi^{-1}(L) \subset Q_{1}$ and $\pi^{-1}(L) \cap Q_{2}$ is a conic since $X$ does not contain planes.
Moreover, since $\pi^{-1}(c) \subset \pi^{-1}(L)\cap Q_{2}$, there is a line $\ell$ such that $\pi^{-1}(L)\cap Q_{2} = \ell \cup \pi^{-1}(c)$ (if $\pi^{-1}(L)\cap Q_{2} = \pi^{-1}(c)$ with multiplicity $2$, then we take $\ell = \pi^{-1}(c)$).
Then we define $\rho(L, c) = \ell$.
Conversely, take a line $\ell \in F(X)$ not passing through $q$.
Since $S\in |H|$, the intersection of $S$ and $\ell$ is a single point, say $p$.
Hence if we take $(\pi(\ell), \pi(p)) \in \mathcal{U}_{C}$, then it maps to the line $\ell$ by $\rho$.
Therefore $\rho$ is a birational morphism and hence $F(X)$ is irreducible.
\end{proof}
\begin{thm}\label{Conics_Quartic}
Let $X$ be a quartic del Pezzo threefold with terminal factorial singularities and let $H$ be the fundamental divisor on $X$.
Then the space $\overline{M}_{0,0}(X, 2)$ parametrizing $H$-conics has two irreducible components.
One component generically parametrizes free birational stable maps, and the other component parametrizes double covers of $H$-lines.
Both components have the expected dimension $4$.
\end{thm}
\begin{proof}
When $X$ is smooth, then Theorem \ref{Conics_Quartic} is proved in \cite{Lehmann2019}. 
Hence we consider the singular case. 

Let $S$, $C$, and $\pi\colon \mathbb{P}^{5}\dashrightarrow \mathbb{P}^{4}$ be as in the proof of Theorem \ref{Lines_Quartic}.
As in the proof of Theorem \ref{Conics_Cubic}, the locus $N\subset \overline{M}_{0,0}(X,2)$ parametrizing double covers of $H$-lines is an irreducible component of dimension $4$.
So we check that $\overline{M}_{0,0}(X,2)\setminus N$ is irreducible.
A general element $(p_{1}, p_{2})\in \mathrm{Sym}^{2}(C)$ defines a plane spanned by $p_{1}, p_{2}$ and $q$.
It induces a rational map $\gamma \colon \mathrm{Sym}^{2}(C) \dashrightarrow \mathbb{G}(2,5)$.
Let $\Gamma$ be the image of this map. 
Consider the set $\Sigma = \{(P, V)\mid P\subset V\}\subset \Gamma \times \mathbb{G}(3,5)$.
The first projection is a $\mathbb{P}^{2}$ bundle.
Thus $\Sigma$ is irreducible since so is $C$.
Then one can construct a birational morphism $\varphi$ from $\Sigma$ to $\overline{M}_{0,0}(X, 2)$ away from the component $N$.
Indeed, if $(\gamma(c_{1}, c_{2}), V)$ is an element of $\Sigma$, then there is a conic $\alpha$ such that $X\cap V = \alpha \cup \pi^{-1}(c_{1}) \cup \pi^{-1}(c_{2})$.
Then we set $\varphi(\gamma(c_{1}, c_{2}), V) = \alpha$.
Conversely, we take an irreducible conic $\alpha$ on $X$ not passing through $q$.
Since $S\in |H|$,  we see that $S\cdot \alpha = 2$.
Hence $\pi(S\cap \alpha)$ can be viewed as an element of $\mathrm{Sym}^{2}(C)$ and this element maps to the plane $P$ by $\gamma$.
On the other hand, let $V$ be the projective $3$-space spanned by $\alpha$ and $q$.
Then, by construction, we see that $\varphi(P, V) = \alpha$.
Therefore the claim holds.
\end{proof}

\begin{cor}
Let $X$ be a cubic or quartic del Pezzo threefold with terminal factorial singularities and let $H$ be the fundamental divisor on $X$.
Then for each integer $d\ge 2$, the space $\overline{M}_{0,0}(X, d)$ parametrizing $H$-degree $d$ curves has two irreducible components.
One component generically parametrizes free birational stable maps, and the other parametrizes $d$-sheeted covers of $H$-lines.
Both components have the expected dimension $2d$.
\end{cor}
\begin{proof}
    One can prove the claim as in \cite{Lehmann2019}*{Theorem 7.9}. 
    We give a proof for completeness.
    When $d=2$, the claim is proved in Theorem \ref{Conics_Cubic} and Theorem \ref{Conics_Quartic}.
    We consider the case $d \ge 3$.
    Let $N_{d} \subset \overline{M}_{0,0}(X, d)$ be the locus parametrizing $d$-sheeted covers of $H$-lines.
    Then by Theorem \ref{Lines_Cubic} and Theorem \ref{Lines_Quartic}, and the dimension count, we see that $N_{d}$ is an irreducible component of $\overline{M}_{0,0}(X, d)$ of dimension $2d$.
    \par
    Let $R_{d} \subset \overline{M}_{0,0}(X, d)$ be another component. 
    Since $X$ does not have subvarieties with higher $a$-invariant, $R_{d}$ generically parametrizes free curves by Lemma \ref{free_iff_dominant}. 
    By dimension count as in the proof of Proposition \ref{multiple_covers}, $M$ generically parametrizes birational maps. 
    Then by Theorem \ref{MBB_noE5} and by induction of $d$,
    $R_{d}$ contains a chain of free curves of $H$-degree $\le 2$.
    Moreover, since any smooth $H$-conic can be deformed to a union of two free $H$-lines, using \cite{Lehmann2019}*{Lemma 5.9}, we see that $R_{d}$ contains a chain $[f]$ of free $H$-lines of length $d$.
    Take a subchain of length $d-1$, which is a smooth point of $\overline{M}_{0,0}(X,d-1)$. 
    Smoothing this chain (this operation is valid after taking a resolution $\tilde{X}\rightarrow X$ by \cite{Kollar1996}*{II.7.6} and the fact that images of free curves on $\tilde{X}$ to $X$ are also free), we obtain a free curve of degree $d-1$, and hence $[f]$ is contained in the locus $\Sigma := \overline{M}_{0,1}(X,1)\times_{X} R'_{d-1} \subset \overline{M}_{0,0}(X,d)$, where $R'_{d-1} \subset \overline{M}_{0,1}(X,d-1)$ is the corresponding component to $R_{d-1}\subset \overline{M}_{0,0}(X,d-1)$.
    Assume that the general fiber of the evaluation map ${s}_{d-1}\colon R'_{d-1} \rightarrow X$ is not irreducible. 
    Then by Proposition \ref{factor through a-cover}, ${s}_{d-1}$ factors rationally through an $a$-cover $f\colon Y\rightarrow X$.
    However, Theorem \ref{a-covers} shows that $f$ factors rationally through $\overline{M}_{0,1}(X,1)\rightarrow X$. 
    This is a contradiction since $R_{d-1} \neq N_{d-1}$. 
    Thus the general fiber of ${s}_{d-1}$ is irreducible.
    Then the locus $\Sigma$ is irreducible by Theorem \ref{Lines_Cubic}, Theorem \ref{Lines_Quartic}, and induction on $d$.
    Since $[f]$ is a smooth point, $R_{d}$ is the unique component containing $\Sigma$: if there is another component $M \neq R_{d}, N_{d}$, then $M$ also contains the locus $\Sigma$ as discussed above, which contradicts the irreducibility of $\Sigma$. 
    Therefore, $\overline{M}_{0,0}(X,d)$ has exactly two components $R_{d}$ and $N_{d}$, completing the proof.
\end{proof}

We are now ready to prove Geometric Manin's Conjecture for our case: 
\newtheorem*{MainThm3}{\rm\bf Theorem \ref{GMC_dP}}
\begin{MainThm3}
    Let $X$ be a factorial del Pezzo threefold with an ample generator $H$.
    Suppose that $H^{3} \ge 3$.
    Then Geometric Manin's Conjecture holds for $X$ with $c = 1$ and $\alpha$ is the class of $H$-conics.
\end{MainThm3}

\begin{rem}
        Note that Geometric Manin's Conjecture is known for smooth del Pezzo threefolds by \cite{Lehmann2019} when $H^{3} \ge 3$ and \cite{Lehmann2019}, \cite{Shimizu2022} when $H^{3} \le 2$ and $X$ is general.
        We also note that any factorial del Pezzo threefold with $H^{3} \ge 5$ is indeed smooth by \cite{Prokhorov2013a}*{Lemma 3.3} and \cite{Prokhorov2017}.
        Hence, it suffices to consider the case when $X$ is singular and $H^{3} = 3, 4$.
 \end{rem}

\begin{proof}[Proof of Theorem \ref{GMC_dP}]
Since $X$ has neither contractible divisors nor $-K_{X}$-lines, $X$ has no subvarieties $Y$ with $a(Y, -K_{X})>1$ by Theorem \ref{highA}.
We also see that if $f\colon Y\rightarrow X$ is an $a$-cover with $Y$ smooth, then $\kappa(Y, K_{Y}-f^{*}K_{X}) = 2$ by Theorem \ref{a-covers}.
\par
Let $F(X)$ be the Fano surface of $H$-lines $\ell$ on $X$, which is isomorphic to $\overline{M}_{0,0}(X, \ell)$ and irreducible by Theorem \ref{Lines_Cubic} and Theorem \ref{Lines_Quartic}.
Let $\mathcal{U}\rightarrow F(X)$ be the universal family with the evaluation morphism $s\colon \mathcal{U}\rightarrow X$.
Take a smooth resolution $\tilde{\mathcal{U}}\rightarrow \mathcal{U}$.
Then the composition $\tilde{s}\colon \tilde{\mathcal{U}}\rightarrow \mathcal{U}\rightarrow X$ is an $a$-cover with $\kappa(\tilde{\mathcal{U}}, K_{\tilde{\mathcal{U}}} - \tilde{s}^{*}K_{X}) = 2$ by Proposition \ref{factor through a-cover} and the fact that $\tilde{s}$ is birational to the Stein factorization.
For each integer $d \ge 1$, let $N_{d} \subset \overline{M}_{0,0}(X,d\ell)$ be the unique component parametrizing $d$-sheeted covers of $H$-lines.
We also take the locus $\tilde{N}_{d} \subset \overline{M}_{0,0}(\tilde{\mathcal{U}})$ parametrizing $d$-sheeted covers of $-K_{\tilde{\mathcal{U}}}$-conics contained in the Iitaka fibration for $K_{\mathcal{U}}-\tilde{s}^{*}K_{X}$. 
This locus is indeed a component of $\overline{M}_{0,0}(\tilde{\mathcal{U}})$ since such $-K_{X}$-conics dominates $\mathcal{U}$ and by dimension count. 
Then $\tilde{s}$ induces a dominant map $\tilde{N}_{d} \dashrightarrow N_{d}$, hence $N_{d}$ are accumulating components.
On the other hand, the other component $R_{d} \subset \overline{M}_{0,0}(X, d\ell)$ generically parametrizing birational stable maps is Manin for each $d\ge 2$ since any $a$-cover factors rationally through $\tilde{s}$ by Lemma \ref{conicbdl}.
Thus the claim holds.
\end{proof}

\bibliography{math}
\end{document}